\documentclass[11pt]{article}
\usepackage{color}
\usepackage[dvips, outer=2cm, inner=2cm,top=2cm,bottom=2cm]{geometry} % See geometry.pdf to learn the layout options. There are lots.
\usepackage[nottoc]{tocbibind}

\usepackage[parfill]{parskip} % Activate to begin paragraphs with an empty line rather than an indent
\usepackage{graphicx}
\usepackage{amsfonts}
\usepackage{amsmath}
\usepackage{amsthm}
\usepackage{amssymb}
\usepackage{epstopdf}
\newtheorem{thm}{Theorem}[section]

\newtheorem{prop}[thm]{Proposition}
\newtheorem{definition}[thm]{Definition}
\newtheorem{lemma}[thm]{Lemma}

\newtheorem{cor}[thm]{Corollary}

\DeclareMathOperator{\tr}{tr}
\DeclareMathOperator{\sym}{Sym}

\date{}

\begin{document}

\title{Beyond Endoscopy via the trace formula - II\\
 Asymptotic expansions of Fourier transforms and bounds towards the Ramanujan conjecture}

\author{S. Ali Altu\u{g}}

\maketitle

\begin{abstract}
We continue the analysis of the elliptic part of the trace formula initiated in \cite{Altug:2015aa}.  In that reference Poisson summation was applied to the elliptic part and the dominant term was analyzed. The main aim of this paper is to study the remaining terms after Poisson summation. We analyze the the Fourier transforms of (smoothed) orbital integrals and obtain exact asymptotic expansions. As an application we recover, using the Arthur-Selberg trace formula, Kuznetsov's result (cf. \cite{Kuznetsov:1980aa}) that the trace of the $p$th Hecke operator on cuspidal automorphic representations is bounded by $p^{\frac14}$.

\end{abstract}

\tableofcontents

\begin{section}{Introduction}

In \cite{Altug:2015aa}, a study of the elliptic part of the trace formula was initiated. By using an appropriate approximate functional equation the elliptic part was rewritten. After an application of Poisson summation, the dominant term (i.e. the term corresponding to the integral of the function) was analyzed, and contributions of certain representations (including the trivial representation) were isolated in this dominant term (cf. theorem 1.1 of \cite{Altug:2015aa}).

The aim of the current paper is to pursue the analysis of \cite{Altug:2015aa} further by studying the rest of the terms that appear after Poisson summation. In the main results of the paper (theorems \ref{thmasymp} and \ref{thmasymp2}) we develop asymptotic expansions, uniform on every variable, of the Fourier transforms of the smoothed orbital integrals that appear in theorem 1.1 of \cite{Altug:2015aa}. These expansions then allow us to control the remaining terms in the elliptic part after the removal of the dominant term.

We also recall that one of the initial goals of \cite{Altug:2015aa} was to get an explicit form of the Arthur-Selberg trace formula that can be used directly in analytic problems involving Hecke eigenvalues. The asymptotic expansions, in particular, establishes this goal for $GL(2)$ and allows us to have good control over the elliptic part.

To demonstrate how the techniques of the paper can be used in practice (and to keep the technical parts of the paper motivated) we use our results and recover the following theorem of Kuznetsov (cf. \cite{Kuznetsov:1980aa}) giving a bound Ramanujan conjecture.  

\begin{thm}\label{thmselberg}Let $f^p$ be as in \S\ref{sectest} and $R_0(f^p)$ denote the right regular representation acting on $L^2_0(Z_+\backslash G(\mathbb{A}))$ (see \S\ref{secautreps} for notation). Then,
\[\tr(R_0(f^p))=O(p^{\frac14}),\]
where the implied constant depends only on the archimedian component of the test function $f^p$.

\end{thm}
We remark that the Ramanujan conjecture would imply that the same bound holds with $O(p^{\epsilon})$ for any $\epsilon>0$. Although theorem \ref{thmselberg} establishes the best bound towards the Ramanujan conjecture proved using \emph{only} the Arthur-Selberg trace formula\footnote{The only prior result seems to be in \cite{Moreno:1977aa}, where the author gets $O(p^{\frac12})$ (the contribution of the trivial representation).}, the importance of it is the method of proof rather than the particular bound. In fact, as mentioned above, the bound itself goes back\footnote{Kuznetsov's article considers Maass forms. For holomorphic forms the same bound follows from the combination of \cite{Kloosterman:1927aa} and Weil's bound on Kloosterman sums in \cite{Weil:1948aa}. We also recall that for holomorphic forms the Ramanujan conjecture is proved in \cite{Deligne:1974aa}.} to Kuznetsov (\cite{Kuznetsov:1980aa}), where he proved his celebrated (relative) trace formula and deduced the bound from this formula, and since then better bounds have been obtained by using different methods\footnote{See \cite{Sarnak:2005aa} for an excellent survey of these type of results.}. We also would like to note that the analysis we present here is very similar to the one of Kuznetsov's. After the Poisson sum, the elliptic part becomes a sum of ``Kloosterman-like" sums (cf. theorem \ref{recallthm}) weighted by certain Fourier transforms (This can be thought  of as a Kuznetsov formula without the weights). Our analysis of these archimedean factors combined with Weil's estimate then gives the bound.

\subsection{Outline of the paper}

In order to separate the technical part of the paper from the proof of \ref{thmselberg} we have divided the paper into two. Sections \S\ref{secnonell} and \S\ref{secell} are dedicated to the proof of theorem \ref{thmselberg}. These sections assume the analysis of Fourier transforms and character sums given in the appendix, and use them to establish the bound. In \S\ref{secasympfourier} we also included a heuristic discussion on how the Fourier transforms behave and what kind of asymptotic behavior one should expect. 

The appendices carry the technical burden of the paper. In appendix \ref{apa} we develop the asymptotic expansions of the smoothed orbital integrals. Most of the work of this appendix goes in to get the expansions uniformly on all of the parameters $\xi,l,f,$ and $p$. The expansions are given in theorems \ref{thmasymp} and \ref{thmasymp2} of \S\ref{singularfourier}, which are fundamental to the whole paper. Finally, in appendix \ref{sectioncharsum} we calculate the local character sums, $Kl_{l,f}(\xi,n)$, and bound them using Weil's estimate on Kloosterman sums. We emphasize that, although presented as an appendix, appendix \ref{apa} forms the backbone of the paper.

\subsection{Notation.}

\begin{itemize}
\item $\mathbb{Z},\mathbb{R}$ and $\mathbb{C}$ as usual will denote the sets of integers, real, and complex numbers respectively. $\mathbb{N}=\{0,1,2,\cdots\}$ will denote the set of natural numbers. 
\item The Mellin transform of a function $\Phi$ is defined as usual, $\tilde{\Phi}(u)=\int_0^{\infty}\Phi(x)x^{u-1}dx$. 

\item $\mathcal{S}(\mathbb{R})$ will denote the Schwartz space, $\mathcal{S}(\mathbb{R})=\begin{setdef}{\Phi\in C^{\infty}(\mathbb{R})}{\sup|x^{\alpha}\Phi^{(\beta)}(x)|<\infty\,\,\,\,\forall \alpha,\beta\in\mathbb{N}}\end{setdef}$. We also remark that for functions $\Phi$ which are only defined on $\mathbb{R}^+$, by abuse of notation, we will use $\Phi\in\mathcal{S}(\mathbb{R})$ to mean that $\Phi(x)$ and all of its derivatives decay faster than any polynomial for $x>0$.
\item For a domain $D\subset \mathbb{C}$, $f=O_h(g)$ means that there exists a constant $K$, depending only on $h$, such that $|f(x)|\leq K |g(x)|$ for every $x\in D$. Most of the times $D$ will be clear from the context and we will not be specified. $f\ll_h g$ means $f=O_h(g)$.
\item $\sqrt{\cdot}$ denotes the branch of the square-root function that is positive on $\mathbb{R}^+$, $\left(\tfrac{D}{\cdot}\right)$ denotes the Kronecker symbol, and $e(x):=e^{2\pi i x}$.

\end{itemize}

\end{section}

\section{Preliminaries and normalizations}\label{secprelim}

This section is a quick review of the first two sections of
\cite{Altug:2015aa}. We first introduce the space of automorphic
representations that will be of interest to us in this paper,
and then review the elliptic part of the trace formula
expressing the trace of the $p$th Hecke operator on this space.
The setup is the one in \S2.1 of \cite{Altug:2015aa}  and we repeat it to keep the text is self-contained. The reader familiar with the normalizations of that reference can skip this section.

\subsection{Automorphic representations}\label{secautreps}
Let $G:=GL(2)$
and $\mathbb{A}=\mathbb{A}_{\mathbb{Q}}$ denote the ring of
ad\`{e}les of $\mathbb{Q}$. Let $Z=Z_G$ denote the center of $G$ and let us denote those matrices in $Z(\mathbb{R})$ having positive entries by $Z_+$. We will be interested in automorphic representations, $\pi$, of $G(\mathbb{A})$  whose central character is trivial on $Z_+$ (identified with $Z_+\simeq \mathbb{R}^+\hookrightarrow \mathbb{A}^{\times}$) and which are unmarried at every \emph{finite} place $v$. We remark that
since we are insisting $\pi_v$ to be unramified at every finite
place and the central character to be trivial on $Z_+$, by
strong approximation, the central character of the
representation $\pi$ is necessarily trivial.

\subsection{Measure normalizations}

On $G$, at a non-archimedian prime $v$ we choose the Haar measure on $G(\mathbb{Q}_v)$ giving measure $1$ to $G(\mathbb{Z}_v)$, and at $\infty$ we choose any Haar measure (the explicit choice is not important for us). For an element $\gamma\in G(\mathbb{Q})$ let the us denote the centralizer of $\gamma$ in $G$ by $G_{\gamma}$. On $G_{\gamma}$ we normalize the measures in a similar manner: 
\begin{itemize}
\item At a non-archimedian prime $v$ we choose the Haar measure giving measure $1$ to $G_{\gamma}(\mathbb{Z}_v)$.
\item At $\infty$, any $\delta\in G(\mathbb{R})$ can be decomposed as $\delta=z_{\delta}\bar{\delta}u_{\delta}$, where $z_{\delta}\in Z_+$ is the central matrix with entries $\sqrt{|\det(\delta)|}$, $u_{\delta}=\left(\begin{smallmatrix}sign(\det(\gamma))&\\&1\end{smallmatrix}\right)$, and $\bar{\delta}\in SL_2(\mathbb{R})$. 

\begin{itemize}

\item If $\gamma\in G(\mathbb{Q})$ that is elliptic over  $\mathbb{R}$ (i.e. has two \emph{non-real} eigenvalues), and let the eigenvalues of $\bar{\delta}\in G_{\gamma}(\mathbb{R})$ be $e^{i\theta},e^{-i\theta}$. We take the measure to be $d\theta$.

\item If $\gamma\in G(\mathbb{Q})$ that is split over $\mathbb{R}$ (i.e. has two distinct \emph{real} eigenvalues), and let the eigenvalues of $\bar{\delta}\in G_{\gamma}(\mathbb{R})$ be $\lambda, \lambda^{-1}$. We take the measure to be $\frac{d \lambda}{\lambda}$.

\end{itemize}
\end{itemize}

\subsection{Test functions}\label{sectest}

Let $p$ be an odd prime. For each finite place $v$ and integer $k\geq0$, let $f_v^{(k)}$ be defined by
\[f_v^{(k)}:=q_v^{-\frac k2}\,\text{characteristic function of }\begin{setdef}{X\in Mat_{2\times2}(\mathbb{Z}_v)}{|\det(X)|_v=q_v^{-k}}\end{setdef},\]
where $q_v$ denotes the cardinality of the residue field. Note that $f_v^{(0)}$ is the unit element of the Hecke algebra and for any admissible representation $\pi_v,$
\[\tr(\pi_v(f_v^{(k)}))=\begin{cases}\tr(\sym^k(A(\pi_v)))& \pi_v\text{ is unramified}\\ 0&\text{otherwise}\end{cases},\]
where $A(\pi_v)\subset {^L}G$ denotes the semi simple conjugacy class parametrizing $\pi_v$ (i.e. the Stake parameter of $\pi_v$). Let $f_{\infty}\in C^{\infty}(Z_+\backslash G(\mathbb{R}))$ be such that its orbital integrals are compactly supported (other than this condition $f_{\infty}$ is arbitrary). We then form the function, $f^p$, by
\[f^p:=f_{\infty}f_p^{(1)}\prod_{\substack{v-finite\\ v\neq p}}f_v^{(0)}.\]
Let $R,R_{disc},$ and $R_0$ denote the right regular representation acting on $L^2$ $(=L^2(Z_+\backslash G(\mathbb{A})))$, $L^2_{disc}, $ and $L_0^2$ respectively, where, as usual, the subscript $``disc"$ denotes the discrete part of the spectrum and the subscript $``0"$ denotes the cuspidal part. For the rest of the paper we will be interested in the trace of the operator $R_0(f^p)$, which is the trace of the $p$th Hecke operator acting on the space of cuspidal automorphic representations satisfying the properties of \S\ref{secautreps} and with the archimedian constraint depending on the function $f_{\infty}$.

\subsection{The trace formula}\label{sectf}

Following \cite{Langlands:2004aa} (pg. 23-34), to calculate the trace of the operator $R_{0}(f^{p})$ we will use the trace formula given on pg. 261-262 of \cite{Jacquet:1970aa}. This formula gives an expression for the trace of the operator $R_{disc}(f^{p})$ as a sum of seven terms
\[\tr(R_{disc}(f))=(i)+(ii)+(iv)+(v)+(vi)+(vii)+(viii).\]
Terms $(i)$ to $(iv)$ are the contribution of the geometric side of the trace formula and the rest are contributions of the continuous part of the spectrum. The term $(iii)$ is omitted because we are working over a number field.

We remark that the above formula gives the trace of the operator \emph{over the whole discrete spectrum} rather than just on the cuspidal part of the spectrum. In other words, in order to get $\tr(R_0(f^{p}))$ we need to subtract from the above formula the contribution of those automoprhic representations that occur discretely but are not cuspidal. For $G=GL(2)$ these are all one dimensional, and in our situation (i.e. over $\mathbb{Q}$, unramified at every finite place, central cahracter trivial on $Z_+$) the only such representation that is contributing to the trace formula is the trivial representation, $\textbf{1}$. Therefore the trace of the operator on the cuspidal part of the spectrum is
\begin{equation}\label{trace}
\tr(R_0(f^{p}))=(ii)-\tr(\textbf{1}(f^{p}))+(i)+(iv)+(v)+(vi)+(vii)+(viii)
\end{equation}
The term $(ii)$ in the above formula is the so-called elliptic term, and its analysis is the most complicated among all. The reason we have grouped the first two terms together in \eqref{trace} is because the term $\tr(\textbf{1}(f^{(p)}))$ appears in the contribution of the elliptic term, and we will show that it actually gives the dominant contribution to the elliptic term (It is of size $\sim p^{\frac12}$ whereas the difference is of size $\sim p^{\frac14}$. cf. theorem \ref{ellipticcont}).

\section{Contribution of the non-elliptic part}\label{secnonell}

In this section we will analyze all but the elliptic term (i.e. $(ii)-\tr(\textbf{1}(f^{p}))$) that appear in \eqref{trace}. The bounds are straightforward and follow from the considerations of \cite{Langlands:2004aa} (pg. 24-28).

\begin{lemma}\label{kuzlem1}
\[(i)=0.\]
\end{lemma}

\begin{proof}By equation $(TF.1)$ of \cite{Langlands:2004aa} we have
\begin{equation*}
(i)=\sum_{z\in Z(\mathbb{Q})}\mu(\mathbb{R}_+G(\mathbb{Q})\backslash G(\mathbb{A}))f^{p}(z)
\end{equation*}
Since $f^{(p)}(z)$ is supported only on those $z\in Z(\mathbb{Q})$ which satisfy $|\det(z)|_{\infty}=p$ and any $z\in Z(\mathbb{Q})$ has determinant that is a square in $\mathbb{Q}$ the sum over $z$ is empty.
\end{proof}

\begin{lemma}\label{kuzlem2}
\[(iv)=O_{f_{\infty}}(1).\]
\end{lemma}

\begin{proof}By the first paragraph of \S2.3 of \cite{Langlands:2004aa} we know that for fixed $f_{\infty}$ this term only contributes for finitely many $p$ (depending on the support of $f_{\infty}$). The lemma follows.

\end{proof}

\begin{lemma}\label{kuzlem3},
\[(v)=0.\]

\end{lemma}

\begin{proof}Follows from the second paragraph of \S2.3 of  \cite{Langlands:2004aa}.

\end{proof}

\begin{lemma}\label{kuzlem4}
\begin{equation*}
(vi)=\int_{|x|>1}\tfrac{\theta_{\infty}(x)}{\sqrt{x^2-1}}dx+\int_{\mathbb{R}}\tfrac{\theta_{\infty}^{-}(x)}{\sqrt{x^2+1}}dx.
\end{equation*}
\end{lemma}

\begin{proof} The contribution of $(vi)$ is given by (31) and (32) of \cite{Langlands:2004aa} ($m=1$ for our case), which is denoted by $\tr(\xi_{0}(f_{\infty}))$ in the same reference. The fact that $\tr(\xi_0(f_{\infty}))$ is the sum of the integrals given above is just a restatement of the second equality of lemma 6.2 of \cite{Altug:2015aa}, where we take $k=1$.

\end{proof}

\begin{lemma}\label{kuzlem5}
\[(vii)=O_{f_{\infty}}(1).\]

\end{lemma}

\begin{proof}By $(33)$, $(34)$ and $(35)$ of \cite{Langlands:2004aa},
\begin{align*}
(vii)&=\tfrac{1}{4\pi}\int_{-i\infty}^{i\infty}\tr(\xi_s(f_{\infty}))(p^{\frac{s}{2}}+p^{\frac{-s}{2}})\left\{-\tfrac{1}{2}\tfrac{\Gamma'\left(\frac{1-s}{2}\right)}{\Gamma\left(\frac{1-s}{2}\right)}-\tfrac{1}{2}\tfrac{\Gamma'\left(\frac{1+s}{2}\right)}{\Gamma\left(\frac{1+s}{2}\right)}-\tfrac{\zeta'(1-s)}{\zeta(1-s)}-\tfrac{\zeta'(1+s)}{\zeta(1+s)}\right\}d|s|\\
&\leq \int_{-i\infty}^{i\infty}|\tr(\xi_s(f_{\infty}))|\left|-\tfrac{1}{2}\tfrac{\Gamma'\left(\frac{1-s}{2}\right)}{\Gamma\left(\frac{1-s}{2}\right)}-\tfrac{1}{2}\tfrac{\Gamma'\left(\frac{1+s}{2}\right)}{\Gamma\left(\frac{1+s}{2}\right)}-\tfrac{\zeta'(1-s)}{\zeta(1-s)}-\tfrac{\zeta'(1+s)}{\zeta(1+s)}\right|d|s|\\
&=O_{f_{\infty}}(1).
\end{align*}

\end{proof}

\begin{lemma}\label{kuzlem6}
\[(viii)=O_{f_{\infty}}(1).\]

\end{lemma}

\begin{proof}From the last paragraph of \S2.6 of \cite{Langlands:2004aa} we see that the contribution of $(viii)$ is given by
\[\tfrac{1}{4\pi}\int_{-i\infty}^{i\infty}\tr(R^{-1}(s)R'(s)\xi_s(f_{\infty}))(p^{s}+p^{-s})d|s|.\]
Again by the same paragraph we also know that,
\[|\tr(R^{-1}(s)R'(s)\xi_s(f_{\infty}))|=O_{f_{\infty}}(s^{-2})\]
We therefore get,
\begin{align*}
(viii)&=\tfrac{1}{4\pi}\int_{-i\infty}^{i\infty}\tr(R^{-1}(s)R'(s)\xi_s(f_{\infty}))(p^{s}+p^{-s})d|s|\\
&=O_{f_{\infty}}\left(\int_{-\infty}^{\infty}(1+s^2)^{-1}ds\right)\\
&=O_{f_{\infty}}(1).
\end{align*}
\end{proof}

Next corollary combines the results of lemmas \ref{kuzlem1} to \ref{kuzlem6} and expresses the contribution of the non-elliptic part of the trace formula. We remark that the integrals below can be absorbed in the $O_{f_{\infty}}(1)$ term, however we are keeping them because they cancel exactly with a part of the contribution of the elliptic term (cf. theorem \ref{recallthm}).

\begin{cor}\label{nonell}The contribution of the non-elliptic part of the trace formula is given by,
\[(i)+(iv)+(v)+(vi)+(vii)+(viii)=\int_{|x|>1}\tfrac{\theta_{\infty}(x)}{\sqrt{x^2-1}}dx+\int_{\mathbb{R}}\tfrac{\theta_{\infty}^{-}(x)}{\sqrt{x^2+1}}dx+O_{f_{\infty}}(1).\]
\end{cor}
\begin{proof}
Follows from lemmas \ref{kuzlem1} to \ref{kuzlem6}.
\end{proof}

\section{Contribution of the elliptic part}\label{secell}

The rest of the paper is dedicated to estimating the contribution of the elliptic part. This section constitutes the main part of the proof and is based on several different components. First, we will use a variant of the main result of \cite{Altug:2015aa} (theorem \ref{recallthm} below) to isolate the contribution of the trivial representation and the contribution of the special representation $\xi_0(f_{\infty})$. We will then use the local analysis given in the appendix (theorems \ref{thmasymp}, \ref{thmasymp2} and corollary \ref{corasymp15}) together with a variant of the Weil bound (corollary \ref{charsumcor3}) to analyze the remaining sums.

\subsection{Poisson summation and special representations}
Let us first recollect the notation and results of \cite{Altug:2015aa} that we will need later on. The main ingredient we will need is theorem \ref{recallthm} below, which is a slight modification of theorem 1.1 of \cite{Altug:2015aa}.

Let $\theta_{\infty}^{pos}(x)$ and $\theta_{\infty}^{neg}(x)$ be defined\footnote{Note that our $\theta_{\infty}^{pos}(x)$ and  $\theta_{\infty}^{neg}(x)$ were respectively denoted by $\theta_{\infty}^+(x)$ and $\theta_{\infty}^-(x)$ in \cite{Altug:2015aa}. The reason for us to choose this notation is to avoid possible confusion with the $\pm$ signs appearing in the asymptotic expansions of theorems \ref{thmasymp} and \ref{thmasymp2}.}  by
\begin{align*}
\theta_{\infty}^{pos}\left(\tfrac{m}{2p^{k/2}}\right)&:=\tfrac{\sqrt{|m^2-4p^k|}}{p^{k/2}}Orb(f_{\infty};\gamma(m,p^k)),\\
\theta_{\infty}^{neg}\left(\tfrac{m}{2p^{k/2}}\right)&:=\tfrac{\sqrt{m^2+4p^k}}{p^{k/2}}Orb(f_{\infty};\gamma(m,-p^k)).
\end{align*}
Where, $\gamma(\alpha,\beta)$ denotes the $GL(2)$ conjugacy class of elements with $\tr=\alpha$ and $\det=\beta$, and $Orb(f_{\infty};\gamma(\alpha,\beta))$ denotes the local orbital integral of the function $f_{\infty}$ at the conjugacy class of $\gamma(\alpha,\beta)$. We remark that the functions are well-defined (i.e. they only depend on the ratio $\frac{m}{\pm2\sqrt{p}}$) because of our choices in \S\ref{secautreps}. For more on this we refer to \S2.2.4 of \cite{Altug:2015aa}.

\subsubsection{Poisson summation.} Theorem \ref{recallthm} below is the starting point of the analysios of the elliptic part. It is essentially theorem 1.1 of \cite{Altug:2015aa} with terms in that theorem explicitly written out. The only difference is in the term $\Sigma(0)$, where instead of moving the contour to $\mathcal{C}_v$ (as is the case in theorem 1.1 of \cite{Altug:2015aa}) we push it to $\Re(u)=0$ (which brings the principal values that appear in the statement).

\begin{thm}[\cite{Altug:2015aa} Theorem 1.1]\label{recallthm} Let $p\neq2$ be a prime, and $f^p$ be defined as in \S\ref{neredeyse}. Then,
\[(ii)=\tr(\textbf{1}(f^{p}))-\tr(\xi_0(f^{p}))-\Sigma(\square)+\Sigma(0)+\Sigma(\xi\neq0).\]
Where, $\tr(\textbf{1}(f^p))$ denotes the contribution of the trivial representation, $\tr(\xi_0(f^p))$ is the contribution of term (vi) to the trace formula (see pg. 25 of \cite{Langlands:2004aa} for notation),
\begin{align*}
\Sigma(\square):&=\left(\theta_{\infty}^{pos}\left(\tfrac{p+1}{2\sqrt{p}}\right)+\theta_{\infty}^{pos}\left(\tfrac{-(p+1)}{2\sqrt{p}}\right)\right)\sum_{f\mid (p-1)}\tfrac{1}{f}\sum_{\substack{l=1\\ \gcd\left(l,\frac{p-1}{f}\right)=1}}^{\infty}\tfrac{1}{l}\left[F\left(\tfrac{lf^2}{p-1}\right)+\tfrac{lf^2}{p-1}H_0\left(\tfrac{lf^2}{p-1}\right)\right]\\
&+\left(\theta_{\infty}^{neg}\left(\tfrac{p-1}{2\sqrt{p}}\right)+\theta_{\infty}^{neg}\left(\tfrac{1-p}{2\sqrt{p}}\right)\right)\sum_{f\mid (p+1)}\tfrac{1}{f}\sum_{\substack{l=1\\ \gcd\left(l,\frac{p+1}{f}\right)=1}}^{\infty}\tfrac{1}{l}\left[F\left(\tfrac{lf^2}{p+1}\right)+\tfrac{lf^2}{p+1}H_0\left(\tfrac{lf^2}{p+1}\right)\right],\\
\Sigma(0):&=\tfrac{1}{2\pi i}\int_{\mathbb{R}}\theta_{\infty}^{pos}(x)\left[\int _{(-1)}\tfrac{\tilde{F}(u)(4p)^{\frac{1+u}{2}}\zeta(2u+2)(1-p^{-(u+1)})}{\zeta(u+2)}|x^2-1|^{\frac u2}du\right]dx\\
&+\tfrac{\sqrt{\pi}}{2\pi i}\int _{|x|>1}\tfrac{\theta_{\infty}^{pos}(x)}{\sqrt{x^2-1}}\left[\mathcal{P}\int_{(0)}\tfrac{\tilde{F}(u)\Gamma\left(\frac{u}{2}\right)(4p)^{\frac u2}\zeta(2u)(1+p^{-u})}{\pi^u\Gamma\left(\frac{1-u}{2}\right)\zeta(u+1)}(x^2-1)^{\frac u2}du\right]dx\\
&+\tfrac{\sqrt{\pi}}{2\pi i}\int _{-1}^1\tfrac{\theta_{\infty}^{pos}(x)}{\sqrt{1-x^2}}\left[\int_{(0)}\tfrac{\tilde{F}(u)\Gamma\left(\frac{1+u}{2}\right)(4p)^{\frac u2}\zeta(2u)(1+p^{-u})}{\pi^u\Gamma\left(\frac{2-u}{2}\right)\zeta(u+1)}(1-x^2)^{\frac u2}du\right]dx\\
&+\tfrac{1}{2\pi i}\int_{\mathbb{R}}\theta_{\infty}^{neg}(x)\left[\int _{(-1)}\tfrac{\tilde{F}(u)(4p)^{\frac{1+u}{2}}\zeta(2u+2)(1-p^{-(u+1)})}{\zeta(u+2)}(x^2+1)^{\frac u2}du\right]dx\\
&+\tfrac{\sqrt{\pi}}{2\pi i}\int _{\mathbb{R}}\tfrac{\theta_{\infty}^{neg}(x)}{\sqrt{x^2+1}}\left[\mathcal{P}\int_{(0)}\tfrac{\tilde{F}(u)\Gamma\left(\frac{u}{2}\right)(4p)^{\frac u2}\zeta(2u)(1+p^{-u})}{\pi^u\Gamma\left(\frac{1-u}{2}\right)\zeta(u+1)}(x^2+1)^{\frac u2}du\right]dx,\\
\Sigma(\xi\neq0):&=\tfrac{\sqrt{p}}{2}\sum_{l,f=1}^{\infty}\tfrac{1}{(lf^2)^{\frac32}}\sum_{\substack{\xi\in\mathbb{Z}\\ \xi\neq0 }}\tfrac{Kl_{l,f}(\xi,p)}{\sqrt{l}}\left[\int_{\mathbb{R}}\theta_{\infty}^{pos}(x)F\left(\tfrac{lf^2(4p)^{-1/2}}{\sqrt{|x^2-1|}}\right)e\left(\tfrac{-x\xi\sqrt{p}}{2lf^2}\right)dx\right]\\
&+\tfrac{1}{4}\sum_{l,f=1}^{\infty}\tfrac{1}{(lf^2)^{\frac12}}\sum_{\substack{\xi\in\mathbb{Z}\\ \xi\neq0 }}\tfrac{Kl_{l,f}(\xi,p)}{\sqrt{l}}\left[\int_{-1}^1\tfrac{\theta_{\infty}^{pos}(x)}{\sqrt{1-x^2}}H_1\left(\tfrac{lf^2(4p)^{-1/2}}{\sqrt{1-x^2}}\right)e\left(\tfrac{-x\xi\sqrt{p}}{2lf^2}\right)dx\right]\\
&+\tfrac{1}{4}\sum_{l,f=1}^{\infty}\tfrac{1}{(lf^2)^{\frac12}}\sum_{\substack{\xi\in\mathbb{Z}\\ \xi\neq0 }}\tfrac{Kl_{l,f}(\xi,p)}{\sqrt{l}}\left[\int_{|x|>1}\tfrac{\theta_{\infty}^{pos}(x)}{\sqrt{x^2-1}}H_0\left(\tfrac{lf^2(4p)^{-1/2}}{\sqrt{x^2-1}}\right)e\left(\tfrac{-x\xi\sqrt{p}}{2lf^2}\right)dx\right]\\
&+\tfrac{\sqrt{p}}{2}\sum_{l,f=1}^{\infty}\tfrac{1}{(lf^2)^{\frac32}}\sum_{\substack{\xi\in\mathbb{Z}\\ \xi\neq0 }}\tfrac{Kl_{l,f}(\xi,-p)}{\sqrt{l}}\left[\int_{\mathbb{R}}\theta_{\infty}^{neg}(x)F\left(\tfrac{lf^2(4p)^{-1/2}}{\sqrt{x^2+1}}\right)e\left(\tfrac{-x\xi\sqrt{p}}{2lf^2}\right)dx\right]\\
&+\tfrac{1}{4}\sum_{l,f=1}^{\infty}\tfrac{1}{(lf^2)^{\frac12}}\sum_{\substack{\xi\in\mathbb{Z}\\ \xi\neq0 }}\tfrac{Kl_{l,f}(\xi,p)}{\sqrt{l}}\left[\int_{\mathbb{R}}\tfrac{\theta_{\infty}^{neg}(x)}{\sqrt{x^2+1}}H_0\left(\tfrac{lf^2(4p)^{-1/2}}{\sqrt{x^2+1}}\right)e\left(\tfrac{-x\xi\sqrt{p}}{2lf^2}\right)dx\right].
\end{align*}
The functions $F,H_0,H_1$ are defined by,
\begin{align*}
F(x):&=\tfrac{1}{2K_0(2)}\int_x^{\infty}e^{-y-\frac1y}\tfrac{dy}{y},\\
H_0(y):&=\tfrac{\sqrt{\pi}}{2 \pi i}\int_{(1)}\tfrac{\Gamma\left(\frac{u}{2}\right)\tilde{F}(u)}{\Gamma\left(\frac{1-u}{2}\right)}(\pi y)^{-u}du,\\
H_1(y):&=\tfrac{\sqrt{\pi}}{2 \pi i}\int_{(1)}\tfrac{\Gamma\left(\frac{1+u}{2}\right)\tilde{F}(u)}{\Gamma\left(\frac{2-u}{2}\right)}(\pi y)^{-u}du,
\end{align*}
where $K_0(z)$ denotes the $0$'th modified Bessel function of the second type\footnote{We actually do not need the specific choice of $F(x)$. Generally, we can take any $F\in \mathcal{S}^{\infty}(\mathbb{R})$ for which the approximate functional equation in proposition 3.4 of \cite{Altug:2015aa} holds.}. The character sums, $K_{l,f}(\xi,\pm p)$, are given by,
\begin{align*}
Kl_{l,f}(\xi,\mp p):&=\sum_{\substack{a\bmod 4lf^2\\ a^2\pm4p\equiv 0\bmod f^2\\ \tfrac{a^2\pm4p}{f^2}\equiv0,1\bmod 4}}\left(\tfrac{(a^2\pm4p)/f^2}{l}\right)e\left(\tfrac{a\xi}{4lf^2}\right).\\
\end{align*}
Finally, $``\mathcal{P}"$ denotes the principal value of an integral (i.e. $\mathcal{P}\int h(u)du=\lim _{\epsilon\rightarrow 0}(\int_\epsilon^{\infty}h(u)du+\int_{-\infty}^{-\epsilon}h(u)du)$).

\end{thm}

\begin{proof}This is, essentially, theorem 1.1 of \cite{Altug:2015aa} with terms explicitly written out. We only need to keep in mind that our functions $\theta_{\infty}^{pos}$ and $\theta_{\infty}^{neg}$ are the functions $\theta_{\infty}^+$ and $\theta_{\infty}^-$  of that theorem, respectively, and in our case the integer denoted by ``$k$'' in the same theorem is $1$ (i.e. $p^k=p$). We now explain each of the terms above. 

\textbf{Note.} \textit{For the rest of the proof, ``theorem 1.1" we will always mean ``theorem 1.1 of  \cite{Altug:2015aa}".}

The expressions for $Kl_{l,f}(\xi,\pm p)$, $H_0$ and $H_1$ are the ones given in theorem 1.1. The term $\Sigma(\xi\neq0)$ is the last sum in the same theorem. The only remaining terms to explain are $\Sigma(\square)$ and $\Sigma(0)$.
\begin{itemize}
\item $\Sigma(\square)$. The expression for $\Sigma(\square)$ given in theorem 1.1 involves a sum over $m\in\mathbb{Z}$ such that $m^2\pm 4p=\alpha^2$ for some $\alpha\in\mathbb{Z}$. Since $p$ is prime and $p\neq2$ it is straightforward to see that the solutions to $m^2-4p=\alpha^2$ are $m^2=(p+1)^2,\alpha^2=(p-1)^2$, and the solutions to $m^2+4p=\alpha^2$ are $m^2=(p-1)^2,\alpha^2=(p+1)^2$. The sum over $f$ in theorem 1.1 is over $f^2\mid (m^2\pm 4p)$ such that $\frac{m^2\pm4p}{f^2}\equiv 0,1\bmod 4$. Since  $m^2\pm4p=(p\pm1)^2$, $f^2\mid (m^2\pm 4p)\Leftrightarrow f\mid (p\pm1)$, moreover since $\frac{(p\pm1)^2}{f^2}$ is a square and any square is congruent to $0,1\bmod 4$, the extra condition on the $f$-sum is vacuous. Finally, the Kronecker symbol, $\left(\frac{(m^2\pm4p)^2/f^2}{l}\right)=\left(\frac{(p\pm 1)^2/f^2}{l}\right)=1$ or $0$ depending on $\gcd(l,(p\pm1)^2/f^2)=1$ or not respectively, which is the same condition as $\gcd(l,(p\pm1)/f)=1$ or not. 

\item $\Sigma(0)$. The expression for this term actually follows from the proof of theorem 6.1 of \cite{Altug:2015aa}, where, by a contour shift the special representations were isolated (see the part of the proof that handles the case $x^2\pm1>0$). There, the curve $\mathcal{C}_v$ was introduced to push the contour of integration to the left of the pole of the integrand at $u=0$, however if one instead pushes the contour to $\Re(u)=0$ one gets the formula for $\Sigma(0)$ given in our theorem. The only point one needs to be careful about is the pole at $u=0$, which, in this case, is on the contour of integration (This is where we get the principal part, $``\mathcal{P}"$, of the integralds.). To keep the paper more self-contained we include the proof of this identity below.

\begin{lemma}\label{lem3.9} Using the notation of theorem 1.1 of \cite{Altug:2015aa}, for any $x\in\mathbb{R}$ we have
\begin{align*}
\tfrac{\sqrt{\pi}}{2\pi i}\int_{\mathcal{C}_v}\left\{\tfrac{\tilde{F}(u)\Gamma\left(\frac{u}{2}\right)(4p)^{\frac u2}\zeta(2u)(1+p^{-u})}{\pi^u\Gamma\left(\frac{1-u}{2}\right)\zeta(u+1)}\right\}(x^2\pm1)^{\frac u2}du&=1+\tfrac{\sqrt{\pi}}{2\pi i}\mathcal{P}\int_{(0)}\left\{\cdots\right\}(x^2\pm1)^{\frac u2}du,\\
\tfrac{\sqrt{\pi}}{2\pi i}\int_{\mathcal{C}_v}\left\{\tfrac{\tilde{F}(u)\Gamma\left(\frac{u}{2}\right)(4p)^{\frac u2}\zeta(2u)(1+p^{-u})}{\pi^u\Gamma\left(\frac{1-u}{2}\right)\zeta(u+1)}\right\}(1-x^2)^{\frac u2}du&=\tfrac{\sqrt{\pi}}{2\pi i }\int_{(0)}\left\{\cdots\right\}(1-x^2)^{\frac u2}du,\\
\end{align*}
where $\{\cdots\}$ on the right of the equality denotes the same function inside the brackets on the left.
\end{lemma}

\begin{proof}

The number $v\in\mathbb{R}$ is chosen so that $\zeta(u+1)$ has no zeros in $|u|<v$ (this, in particular, implies that $v<\frac12$), and $\mathcal{C}_{v}=(-i\infty,iv)\cup(iv,i\infty)\cup \begin{setdef}{ve^{it}}{\frac \pi 2\leq t \leq \frac{3\pi}{2}}\end{setdef}$. Let $D_v$ denote the half-disc $\begin{setdef}{u\in\mathbb{C}}{|u|\leq v, \Re(u)\leq 0}\end{setdef}$. It is straightforward to verify that for any function $h(u)$ that is holomorphic in $D_v$ with at most a simple pole at $u=0$ satisfies
\[\lim_{v\rightarrow 0}\tfrac{1}{2\pi i}\int _{\mathcal{C}_v}h(u)du=\frac{-Res_{u=0}h(u)}{2}+\tfrac{1}{2\pi i}\mathcal{P}\int _{(0)}h(u)du.\tag{*}\label{thm3.8*}\]
By lemma 3.3 of \cite{Altug:2015aa}, $\tilde{F}(u)$ has a simple pole with residue $1$ at $u=0$ (and is holomorphic elsewhere). Note also that $\Gamma(u/2)$ has a simple pole at $u=0$ with residue $2$, and $\zeta(u+1)$ has a pole at $u=0$ with residue $1$, and both are holomorphic in $\begin{setdef}{|u|\leq v}{u\neq0}\end{setdef}$. Therefore around $0$ we have,
\begin{align*}
\tfrac{\tilde{F}(u)\Gamma\left(\frac{u}{2}\right)(4p)^{\frac u2}\zeta(2u)(1+p^{-u})}{\pi^u\Gamma\left(\frac{1-u}{2}\right)\zeta(u+1)}(x^2\pm1)^{\frac u2}&=\tfrac{-2}{\sqrt{\pi}}\tfrac{1}{u}+O(1),\tag{**}\label{thm3.8**}\\
\tfrac{\tilde{F}(u)\Gamma\left(\frac{u}{2}\right)(4p)^{\frac u2}\zeta(2u)(1+p^{-u})}{\pi^u\Gamma\left(\frac{1-u}{2}\right)\zeta(u+1)}(x^2-1)^{\frac u2}&=O(1)\tag{***}\label{thm3.8***}.
\end{align*}
Both of the functions in \eqref{thm3.8**} and \eqref{thm3.8***} are holomorphic in $D_v\backslash \{0\}$. The holomorphy of \eqref{thm3.8***} immediately implies the second equality. Finally, using \eqref{thm3.8*} and \eqref{thm3.8**} we get the first.

\end{proof}
Coming back to the proof of theorem \ref{recallthm}, we substitute the expressions of lemma \ref{lem3.9} into the term with the $\mathcal{C}_v$-integrals in theorem 1.1. The term coming from the $1$ in the first formula of lemma \ref{lem3.9} cancels the the term $-\tfrac{k+1}{2}\sum_{\pm}\{\cdots\}$ of theorem 1.1. This finishes the proof.
\end{itemize}

\end{proof}

As an immediate corollary to theorem \ref{recallthm} we get the following estimate.

\begin{cor}\label{recallcor} Let $p\neq2$ be a prime, $f^p$ be defined as in \S\ref{sectf}. Then, 
\[\tr(R_0(f^p))=\Sigma(\xi\neq0)+O_{f_{\infty}}(\log^2(p)),\]
where $\Sigma(\xi\neq0)$ is as defined in theorem \ref{recallthm}. 
\end{cor}

\begin{proof}By \eqref{trace} and corollary \ref{nonell} we have
\begin{align*}
\tr(R_0(f^p))&= (ii)-\tr(\textbf{1}(f^p))+(i)+(iv)+(v)+(vi)+(vii)+(viii)\\
&=(ii)-\tr(\textbf{1}(f^p))+\int_{|x|>1}\tfrac{\theta_{\infty}(x)}{\sqrt{x^2-1}}dx+\int_{\mathbb{R}}\tfrac{\theta_{\infty}^{-}(x)}{\sqrt{x^2+1}}dx+O_{f_{\infty}}(1).
\end{align*}
By lemma 6.2 of \cite{Altug:2015aa},
\[\tr(\xi_0(f^p))=\int_{|x|>1}\tfrac{\theta_{\infty}(x)}{\sqrt{x^2-1}}dx+\int_{\mathbb{R}}\tfrac{\theta_{\infty}^{-}(x)}{\sqrt{x^2+1}}dx.\]
Substituting the result of theorem \ref{recallthm} into the expression for the contribution of $(ii)$ and using the above equation for $\tr(\xi_0(f^p))$ gives,
\[\tr(R_0(f^p))=-\Sigma(\square)+\Sigma(0)+\Sigma(\xi\neq0)+O_{f_{\infty}}(1).\tag{$\circ$}\label{cor3.10circ}\]
For the contribution of $\Sigma(0)$, note that
\begin{align*}
\int _{(-1)}\tfrac{\tilde{F}(u)(4p)^{\frac{1+u}{2}}\zeta(2u+2)(1-p^{-(u+1)})}{\zeta(u+2)}|x^2-1|^{\frac u2}du&\ll\int_{(-1)}|\tilde{F}(u)||x^2-1|^{\frac u2}du,\\
\mathcal{P}\int_{(0)}\tfrac{\tilde{F}(u)\Gamma\left(\frac{u}{2}\right)(4p)^{\frac u2}\zeta(2u)(1+p^{-u})}{\pi^u\Gamma\left(\frac{1-u}{2}\right)\zeta(u+1)}(x^2-1)^{\frac u2}du&\ll\mathcal{P}\int_{(0)}|\tilde{F}(u)||x^2-1|^{\frac u2}du.
\end{align*}
Therefore, 
\[\Sigma(0)=O_{f_{\infty}}(1).\tag{$\circ\circ$}\label{cor3.10*}\]

Finally, to handle the contribution of $\Sigma(\square)$ we rewrite it using Mellin inversion. Following the steps of proposition 3.1 of \cite{Altug:2015aa} in reverse order we get,
\begin{align*}
&\sum_{f\mid (p\pm1)}\tfrac{1}{f}\sum_{\substack{l=1\\ \gcd\left(l,\frac{p\pm1}{f}\right)=1}}^{\infty}\tfrac{1}{l}\left[F\left(\tfrac{lf^2}{p\pm1}\right)+\tfrac{lf^2}{p\pm1}H_0\left(\tfrac{lf^2}{p\pm1}\right)\right]\\
&=\tfrac{1}{2\pi i}\int_{(2)}\tilde{F}(u)(p\pm1)^uL(u+1,(p\pm1)^2)du-\tfrac{1}{2\pi i}\int_{(-2)}\tilde{F}(u)(p\pm1)^{u}L(u+1,(p\pm1)^2)du,\tag{$\dagger$}\label{cor3.10dagger}
\end{align*}
where 
\[L(u,(p\pm1)^2)=\zeta(u)\sum_{f\mid (p\pm1)}\tfrac{1}{f^{2u-1}}\prod_{q\mid \frac{p\pm1}{f}}\left(1-\tfrac{1}{q^u}\right).\]
Note that $L(u+1,(p\pm1)^2)$ has a simple pole at $u=0$. Furthermore, by lemma 3.3 of \cite{Altug:2015aa}, $\tilde{F}(u)$ also has a simple pole at at $u=0$, therefore the integrand of \eqref{cor3.10dagger} has a double pole. Hence, by the Cauchy integral formula, 
\begin{align*}
\eqref{cor3.10dagger}&\ll  \log(p)\sum_{f\mid (p\pm1)}\tfrac1f\\
&\ll \log^2(p).\tag{$\dagger\dagger$}\label{cor3.10daggerdagger}
\end{align*}
Note that the implied constant above depends only on $F(u)$. Substituting \eqref{cor3.10daggerdagger} in $\Sigma(\square)$ finally gives
\[\Sigma(\square)=O_{f_{\infty}}(\log^2(p)).\tag{$\circ\circ\circ$}\label{cor3.10**}\]
Combining \eqref{cor3.10circ}, \eqref{cor3.10*}, and \eqref{cor3.10**} finishes the proof of the corollary.

\end{proof}

\subsection{The contribution of $\Sigma(\xi\neq0)$}

By corollary \ref{recallcor}, proving theorem \ref{thmselberg} is reduced to bounding $\Sigma(\xi\neq0)$. For the rest of the section we will prove the following theorem:

\begin{thm}\label{ellipticcont}
\[\Sigma(\xi\neq0)\ll p^{\frac14},\]
where the implied constant depends only on $f_{\infty}$.

\end{thm}

To keep the reader oriented, before giving the details of the estimates we present the proof assuming the necessary estimates.

\begin{proof} By corollary \ref{finalanalyticest} to estimate $\Sigma(\xi\neq0)$ it is enough to estimate the two sums that are in that corollary. The theorem then follows from the bounds in propositions \ref{propfirstest} and \ref{propsecondest}.

\end{proof}

The results referred to in the above proof are are obtained from a detailed study of the asymptotic behavior of the Fourier transforms of orbital integrals appearing in $\Sigma(\xi\neq0)$ (cf. theorems \ref{thmasymp} and \ref{thmasymp2}) together with a variant of Weil's estimate on the character sums of $Kl_{l.f}(\xi,\pm p)$ (cf. corollary \ref{charsumcor3}). We present the details of these in appendices to separate the main argument from the technical discussions. We note, however, that the discussions and results of the appendix are central to the whole paper.

We remark that the analysis of the Fourier transforms is delicate because the functions $|1-x^2|^{\frac{a}{2}}\theta_{\infty}^{pos}(x)$ (for $a=0,\pm1$) \emph{are not smooth}. In order to describe the asymptotic behavior we first need to discuss these singularities.

\subsubsection{Singularities of orbital integrals}\label{secsingularities}

Our main references for this section are \S2.2.3, \S2.2.4 of \cite{Altug:2015aa} and \S2.1 of \cite{Langlands:2013aa} (\S2.1). More can be found in \cite{Langlands:2004aa} (pg.21), \cite{Knapp:2001aa} (chapter XI), and \cite{Shelstad:1979aa}.

By equation $(\star\star\star)$ in \S2.2.4 of \cite{Altug:2015aa}\footnote{In \cite{Altug:2015aa} the functions $\theta_{\infty}^{pos},\theta_{\infty}^{neg},g_1^{pos},g_2^{pos}$ were denoted by $\theta_{\infty}^+, \theta_{\infty}^-g_1^+$, and $g_2^+$ respectively.} we know that $\theta_{\infty}^{neg}(x)\in C_c^{\infty}(\mathbb{R})$, and 
\begin{equation}\label{sectriv2}
\theta_{\infty}^{pos}(x)=2\sqrt{|x^2-1|}g_{1}^{pos}(x)+g_2^{pos}(x),
\end{equation}
where $g_1^{pos}(x)$ is supported in $[-1,1]$ and is smooth inside and up to the boundary, $g_2^{\text{ pos}}(x)\in C_c^{\infty}(\mathbb{R})$. Since we will be using the functions $2\sqrt{|x^2-1|}g_1^{pos}(x)$ rather than the function $g_1(x)$ for the rest of the section, following lemma 2.1.4 of \cite{Langlands:2013aa}, let us introduce the notation
\begin{equation}\label{thetajnot}
\theta_{\infty,1}^{pos}(x):=2\sqrt{|x^2-1|}g_1^{pos}(x)\hspace{0.5in},\hspace{0.5in}\theta_{\infty,2}^{pos}(x):=g_2^{pos}(x).
\end{equation}
With this notation,
\begin{equation}\label{thetajnot2}
\theta_{\infty}^{pos}(x)=\theta_{\infty,1}^{pos}(x)+\theta_{\infty,2}^{pos}(x).
\end{equation}
By lemma 2.1.4 of \cite{Langlands:2013aa}\footnote{Strictly speaking, the asymptotic expansions given in lemma 2.1.4 of \cite{Langlands:2013aa} are for the stable orbital integrals on $SL_2(\mathbb{R})$, however the same considerations and expansions go through verbatim to our case too. The relevant part is the decomposition of $\theta_{\infty}^{pos}$ given in \eqref{sectriv2}. We also note that our parameter $``x"$ is the negative of the parameter used in emma 2.1.4 of \cite{Langlands:2013aa}.} we have the following asymptotic expansions\footnote{For a brief discussion of asymptotic expansions see \S\ref{apsecasympexp}.} around $|x\mp1|<\kappa$ 
\begin{equation}\label{series}
\theta_{\infty,1}^{pos}((\pm(1-x))\sim|x|^{\frac12}\sum_{j=0}^{\infty}a_{j}^{\pm}x^j\hspace{0.3in},\hspace{0.3in}\theta_{\infty,2}^{pos}(\pm(1-x))\sim \sum_{j=0}^{\infty}b_j^{\pm}x^j,
\end{equation}
where $``a_j^+,b_j^+"$ are the coefficients for the expansion around $x=1$ and $``a_j^-,b_j^-"$ are the ones for the expansion around $x=-1$. 
\subsubsection{Asymptotic behavior of Fourier transforms}\label{secasympfourier}

In this section we will discuss the asymptotic behavior of the Fourier transforms that appear in $\Sigma(\xi\neq0)$. Since the results are fundamental for the arguments to follow, and may not be very transparent on a first read, before stating them we will discuss the problems, what kind of asymptotic behavior one should expect, and how these will be useful in our application. 

The first and most basic point to note is that we need the asymptotic behavior of these Fourier transforms \emph{uniformly} in all of the variables $\xi,l,f$, and $p$. In other words no implied constant or error bound should not depend on these variables since our aim is to use these results in the sums over $l,f$, and $\xi$, and to understand the size of these sums in terms of $p$. Almost all of the work in appendix \ref{apa} is goes in to get this uniformity in the estimates.

To describe the other aspect of the problem let us begin by remarking that although \emph{the product of the functions} appearing in the Fourier transforms is smooth (cf. proposition 4.1 of \cite{Altug:2015aa}), $\theta_{\infty}^{pos}(x)$ has singularities at $x=\pm1$ of the type described in \eqref{sectriv2}. What makes the product smooth is the rapid decay of the smoothing functions $F$, $H_0$, and $H_1$. Since each has $|1-x^2|^{-\frac12}$ in their arguments, as $x\rightarrow \pm1$ the arguments go to $\infty$ and the functions (and their derivatives) decay to make the product smooth. The exact rate at which this smoothing occurs is reflected in the asymptotic expansions.

\paragraph{A heuristic discussion:}Since the integrals are problematic only around the singularities (outside the singularities an integration by parts argument gives the desired estimate (cf. lemma \ref{aplem4})) we can consider these integrals only around the singularities. Say $x\sim1$, and change the variables to $u=x-1$ then the integrands that appear in $\Sigma(\xi\neq0)$ are all of the form
\[e\left(\tfrac{\xi\sqrt{p}}{2lf^2}\right)\int_{u\sim 0}|u|^{\frac{a}{2}}\Phi\left(\tfrac{lf^2(4p)^{-1/2}}{\sqrt{|u|}}\right)e\left(\tfrac{u\xi\sqrt{p}}{2lf^2}\right)du,\]
where $a\in \{\pm1,0\}$ and $\Phi$ and all its derivatives are rapidly decreasing (cf. the equation right before (*) of theorem \ref{thmasymp}). Ignoring the sign of $\xi$ for the moment, using the change of variables $u\mapsto \frac{2lf^2 u}{\xi\sqrt{p}}$, we (up to a constant) get
\begin{equation}\label{heuristic2}
e\left(\tfrac{\xi\sqrt{p}}{2lf^2}\right)\left(\tfrac{lf^2}{\xi\sqrt{p}}\right)^{1+\frac a2}\int_{u\sim 0}|u|^{\frac{a}{2}}\Phi\left(\sqrt{\tfrac{lf^2\xi p^{-1/2}}{|u|}}\right)e\left(-u\right)du.
\end{equation}
Now, the integral above tells us what to expect from the asymptotic behavior. There is a distinction between the cases for which $\frac{lf^2\xi}{\sqrt{p}}\gg1$ and $\ll1$ (We remark that $\frac{lf^2\xi}{\sqrt{p}}$ is the quantity that is denoted by $C^2D$ in theorems \ref{thmasymp} and \ref{thmasymp2}.). In the former case the rapid decay of $\Phi$ makes the integral very small (cf. propositions \ref{propregion1} and \ref{lemasymp3}), however in the latter case $\Phi$ behaves like a constant (cf. proposition \ref{lemasymp4}) and the integral decays only up to the power of $1+\frac a2$ (cf. theorems \ref{thmasymp}, \ref{thmasymp2}, and corollary \ref{corasymp15}), but oscillates with a very high frequency of $\frac{\xi\sqrt{p}}{lf^2}$. For proving theorem \ref{ellipticcont} we will not need this oscillation and will only be using the leading term $(\frac{lf^2}{\xi\sqrt{p}})^{1+\frac a2}$ (cf. proposition \ref{propregion2}). We hope that this heuristic explanation provides some intuition about the bounds to follow.

The next result is just a simple consequence of the choice of the smoothing function $F\in\mathcal{S}(\mathbb{R})$. We are stating it as a separate lemma for reference in the proofs of the upcoming propositions.

\begin{lemma}\label{asymplem3.11}The functions $F,H_0,H_1$ that appear in the integrals of $\Sigma(\xi\neq0)$ are all in $\mathcal{S}(\mathbb{R})$.
\end{lemma}
\begin{proof}  It is clear from the definition (cf. theorem \ref{recallthm}) that $F$ and all its derivatives decay faster than any polynomial. For $H_0$ and $H_1$ we just need to note that the integral transforms that define $H_0(y)$ and $H_1(y)$ that are given in theorem \ref{recallthm} converge absolutely and define $C^{\infty}$ functions of $y$. Moreover the integrands are holomorphic for $\Re(u)>0$ so we can move the contour freely in the right half-plane. The result then follows from differentiating under the integral sign.

\end{proof}

The following proposition handles the Fourier transforms of $\theta_{\infty}^{neg}$.

\begin{prop}\label{smoothpropasymp}Let $l,f,p,\xi\in\mathbb{Z}\backslash \{0\}$. Then for any $M,N_1\in \mathbb{Z}_{\geq0}$,
\begin{align*}
\int_{\mathbb{R}}\theta_{\infty}^{neg}(x)F\left(\tfrac{lf^2(4p)^{-1/2}}{\sqrt{x^2+1}}\right)e\left(\tfrac{-x\xi\sqrt{p}}{2lf^2}\right)dx&\ll \tfrac{1}{\xi^M}\left(\tfrac{lf^2}{\sqrt{p}}\right)^{M-N_1}\\
\int_{\mathbb{R}}\tfrac{\theta_{\infty}^{neg}(x)}{\sqrt{x^2+1}}H_0\left(\tfrac{lf^2(4p)^{-1/2}}{\sqrt{x^2+1}}\right)e\left(\tfrac{-x\xi\sqrt{p}}{2lf^2}\right)dx&\ll\tfrac{1}{\xi^M}\left(\tfrac{lf^2}{\sqrt{p}}\right)^{M-N_1},
\end{align*}
where the implied constants depend only on $\theta_{\infty}^{neg},F, M$, and $N_1$.
\end{prop}

\begin{proof} By lemma \ref{asymplem3.11} $F$ and $H_0$, satisfy the conditions of \ref{cortoaplem4}. By \S\ref{secsingularities} we know that $\theta_{\infty}^{neg}(x)\in C_c^{\infty}(\mathbb{R})$, therefore the hypotheses of corollary \ref{cortoaplem4} are satisfied (with $a=0$ for the first integral and $a=-1$ for the second,  $C=\frac{lf^2}{\sqrt{4p}}$, and $D=\frac{-\xi\sqrt{p}}{2lf^2}$). The proposition follows from the same corollary.

\end{proof}

We move on to $\theta_{\infty}^{pos}$. The next proposition and its corollary, although is valid for every $l,f$, and $\xi$, will be used for estimating the sums in the range $\frac{lf^2\xi}{\sqrt{p}}\gg1$.
\begin{prop}\label{propregion1}Let $l,f,\xi\in\mathbb{Z}\backslash\{0\}$. Then, for every $N_1\geq 0$,
\begin{align*}
\int_{-1}^1\theta_{\infty,1}^{pos}(x)F\left(\tfrac{lf^2(4p)^{-1/2}}{\sqrt{|x^2-1|}}\right)e\left(\tfrac{-x\xi\sqrt{p}}{2lf^2}\right)dx&\ll\left(\tfrac{lf^2}{\sqrt{p}}\right)^{\frac32-N_1}\tfrac{1}{\xi^{\frac32+N_1}}\tag{i}\label{prop3.13i}\\
\int_{\mathbb{R}}\theta_{\infty,2}^{pos}(x)F\left(\tfrac{lf^2(4p)^{-1/2}}{\sqrt{|x^2-1|}}\right)e\left(\tfrac{-x\xi\sqrt{p}}{2lf^2}\right)dx&\ll\left(\tfrac{lf^2}{\sqrt{p}}\right)^{1-N_1}\tfrac{1}{\xi^{1+N_1}}\tag{ii}\label{prop3.13ii}\\
\int_{-1}^1\tfrac{\theta_{\infty,1}^{pos}(x)}{\sqrt{1-x^2}}H_1\left(\tfrac{lf^2(4p)^{-1/2}}{\sqrt{1-x^2}}\right)e\left(\tfrac{-x\xi\sqrt{p}}{2lf^2}\right)dx&\ll\left(\tfrac{lf^2}{\sqrt{p}}\right)^{1-N_1}\tfrac{1}{\xi^{1+N_1}}\tag{iii}\label{prop3.13iii}\\
\int_{-1}^1\tfrac{\theta_{\infty,2}^{pos}(x)}{\sqrt{1-x^2}}H_1\left(\tfrac{lf^2(4p)^{-1/2}}{\sqrt{1-x^2}}\right)e\left(\tfrac{-x\xi\sqrt{p}}{2lf^2}\right)dx&\ll\left(\tfrac{lf^2}{\sqrt{p}}\right)^{\frac12-N_1}\tfrac{1}{\xi^{\frac12+N_1}}\tag{iv}\label{prop3.13iv}\\
\int_{|x|>1}\tfrac{\theta_{\infty,2}^{pos}(x)}{\sqrt{x^2-1}}H_0\left(\tfrac{lf^2(4p)^{-1/2}}{\sqrt{x^2-1}}\right)e\left(\tfrac{-x\xi\sqrt{p}}{2lf^2}\right)dx&\ll\left(\tfrac{lf^2}{\sqrt{p}}\right)^{\frac12-N_1}\tfrac{1}{\xi^{\frac12+N_1}}\tag{v}\label{prop3.13vi},
\end{align*}
where the implied constants depend only on $\theta_{\infty}^{pos},F$, and $N_1$.
\end{prop}

\begin{proof} The result is a direct application of corollary \ref{corasymp15} where we take $C=\frac{lf^2}{\sqrt{4p}}$ and $D=\frac{-\xi\sqrt{p}}{2lf^2}$. We just need to verify that the functions satisfy the conditions given in the corollary. By lemma \ref{asymplem3.11} $F,H_0$, and $H_1$ satisfy the conditions of corollary. For \eqref{prop3.13i} take the function $h(x)$ of corollary \ref{corasymp15} to be $\theta_{\infty,1}^{pos}(x)$. By \eqref{series} we see that this function satisfies the necessary asymptotic expansion of corollary \ref{corasymp15} with $a=1$. This implies \eqref{prop3.13i}. 

For \eqref{prop3.13ii} take the function $h(x)$ of corollary \ref{corasymp15} to be $\theta_{\infty,2}^{pos}(x)$. Again by \eqref{series} we see that this function satisfies the necessary asymptotic expansion of corollary \ref{corasymp15}, this time with $a=0$. This implies \eqref{prop3.13ii}. 

For the remaining cases just note that \eqref{series} combined with lemma \ref{aplema5} we know $\theta_{\infty,1}^{pos}(x)|x^2-1|^{\frac{-1}{2}}$ satisfies the series expansion of corollary \ref{corasymp15} with $a=1$, and $\theta_{\infty,2}^{pos}(x)|x^2-1|^{\frac{-1}{2}}$ satisfies it with $a=\frac{-1}{2}$. The proposition follows.

\end{proof}

\begin{cor}\label{corregion1}Let $l,f,\xi\in\mathbb{Z}\backslash\{0\}$. Then, for every $N\geq 0$,
\begin{align*}
\int_{\mathbb{R}}\theta_{\infty}^{pos}(x)F\left(\tfrac{lf^2(4p)^{-1/2}}{\sqrt{|x^2-1|}}\right)e\left(\tfrac{-x\xi\sqrt{p}}{2lf^2}\right)dx&\ll\left(\tfrac{\sqrt{p}}{lf^2\xi}\right)^{N}\tfrac{1}{\xi^{2}}\\
\tfrac{lf^2}{\sqrt{p}}\int_{|x|<1}\tfrac{\theta_{\infty}^{pos}(x)}{\sqrt{x^2-1}}H_1\left(\tfrac{lf^2(4p)^{-1/2}}{\sqrt{x^2-1}}\right)e\left(\tfrac{-x\xi\sqrt{p}}{2lf^2}\right)dx&\ll\left(\tfrac{\sqrt{p}}{lf^2\xi}\right)^{N}\tfrac{1}{\xi^{2}}\\
\tfrac{lf^2}{\sqrt{p}}\int_{|x|>1}\tfrac{\theta_{\infty}^{pos}(x)}{\sqrt{x^2-1}}H_0\left(\tfrac{lf^2(4p)^{-1/2}}{\sqrt{x^2-1}}\right)e\left(\tfrac{-x\xi\sqrt{p}}{2lf^2}\right)dx&\ll\left(\tfrac{\sqrt{p}}{lf^2\xi}\right)^{N}\tfrac{1}{\xi^{2}},
\end{align*}
where the implied constants depend only on $\theta_{\infty}^{pos},F$, and $N$.
\end{cor}

\begin{proof} Recall from \eqref{thetajnot2} that $\theta_{\infty}^{pos}=\theta_{\infty,1}^{pos}+\theta_{\infty,2}^{pos}$ and $\theta_{\infty,1}^{pos}(x)$ is supported in $|x|\leq1$. The corollary follows from the following choices of $N_1$ in proposition \ref{propregion1}:
\[\eqref{prop3.13i},\,\eqref{prop3.13iv},\,\eqref{prop3.13vi}:N_1=\tfrac32+N\hspace{0.4in},\hspace{0.4in}\eqref{prop3.13ii},\,\eqref{prop3.13iii}:N_1=2+N.\]
\end{proof}

The following proposition and its corollary handles the remaining case of $\tfrac{lf^2\xi}{\sqrt{p}}\ll1$.

\begin{prop}\label{propregion2}Let $l,f,\xi,n\in\mathbb{Z}\backslash\{0\}$ such that $\tfrac{lf^2\xi}{\sqrt{p}}\ll1$. Then, 
\begin{align*}
\int_{-1}^1\theta_{\infty,1}^{pos}(x)F\left(\tfrac{lf^2(4p)^{-1/2}}{\sqrt{|x^2-1|}}\right)e\left(\tfrac{-x\xi\sqrt{p}}{2lf^2}\right)dx&\ll\left(\tfrac{lf^2}{\sqrt{p}\xi}\right)^{\frac32}\tag{i}\label{prop3.15i}\\
\int_{\mathbb{R}}\theta_{\infty,2}^{pos}(x)F\left(\tfrac{lf^2(4p)^{-1/2}}{\sqrt{|x^2-1|}}\right)e\left(\tfrac{-x\xi\sqrt{p}}{2lf^2}\right)dx&\ll\left(\tfrac{lf^2}{\sqrt{p}}\right)^2\tag{ii}\label{prop3.15ii}\\
\int_{-1}^1\tfrac{\theta_{\infty,1}^{pos}(x)}{\sqrt{1-x^2}}H_1\left(\tfrac{lf^2(4p)^{-1/2}}{\sqrt{1-x^2}}\right)e\left(\tfrac{-x\xi\sqrt{p}}{2lf^2}\right)dx&\ll\tfrac{lf^2}{\sqrt{p}\xi}\tag{iii}\label{prop3.15iii}\\
\int_{-1}^1\tfrac{\theta_{\infty,2}^{pos}(x)}{\sqrt{1-x^2}}H_1\left(\tfrac{lf^2(4p)^{-1/2}}{\sqrt{1-x^2}}\right)e\left(\tfrac{-x\xi\sqrt{p}}{2lf^2}\right)dx&\ll\left(\tfrac{lf^2}{\sqrt{p}\xi}\right)^{\frac12}\tag{iv}\label{prop3.15iv}\\
\int_{|x|>1}\tfrac{\theta_{\infty,2}^{pos}(x)}{\sqrt{x^2-1}}H_0\left(\tfrac{lf^2(4p)^{-1/2}}{\sqrt{x^2-1}}\right)e\left(\tfrac{-x\xi\sqrt{p}}{2lf^2}\right)dx&\ll\left(\tfrac{lf^2}{\sqrt{p}\xi}\right)^{\frac12}\log\left(\tfrac{lf^2\xi}{\sqrt{p}}\right),\tag{v}\label{prop3.15vi}
\end{align*}
where the implied constants depend only on $\theta_{\infty}^{pos}$ and $F$.
\end{prop}

\begin{proof}Let us first recollect the properties of the integrands we will be using.
\begin{enumerate}
\item By the asymptotic expansions of \eqref{series} and lemma \ref{aplema5} we have 
\begin{equation*}
\tfrac{\theta_{\infty,1}^{pos}((\pm(1-x))}{\sqrt{|x^2-2x|}}\sim\sum_{j=0}^{\infty}c_{j}^{\pm}x^j\hspace{0.3in},\hspace{0.3in}\tfrac{\theta_{\infty,2}^{pos}(\pm(1-x))}{\sqrt{|x^2-2x|}}\sim |x|^{-\frac12}\sum_{j=0}^{\infty}d_j^{\pm}x^j,
\end{equation*}
for some constants $c_j^{\pm}$ and $d_j^{\pm}$. 
\item $F, H_0$, and $H_1$ are all in $\mathcal{S}(\mathbb{R})$ (cf. lemma \ref{asymplem3.11}). 
\item $\tilde{F}(u)$ is holomorphic except for a simple pole at $u=0$ (cf. lemma 3.3 of \cite{Altug:2015aa}).
\item $\tilde{H}_1(u)$ is holomorphic in $\Re(u)>-1$ with only a simple pole at $u=0$, whereas $\tilde{H}_0(u)$ is holomorphic in $\Re(u)>-2$ with a double pole at $u=0$ (This follows from the definition of $H_0$ and $H_1$ together with the holomorphy of $\tilde{F}$.).
\end{enumerate}
The proof is now a direct application of theorems \ref{thmasymp}, \ref{thmasymp2}, and corollary \ref{corasymp15}. In all the applications $C=\tfrac{lf^2}{2\sqrt{p}}$ and $D=\tfrac{-\xi\sqrt{p}}{2lf^2}$. We will give the details for \eqref{prop3.15i} and describe the modifications for the other cases.
\begin{itemize}
\item \eqref{prop3.15i}. By \eqref{series}, we can apply theorem \ref{thmasymp} with $a=1$, $M=0$, $\tau_1=0$, $\Phi=F$, and $h_1=\theta_{\infty,1}^{pos}$. Since $\tfrac{lf^2\xi}{\sqrt{p}}\ll1$ this gives,
\[\eqref{prop3.15i}\ll \left(\tfrac{lf^2}{\xi\sqrt{p}}\right)^{\frac32}|\mathcal{A}_{\theta_{\infty,1}^{pos},0}^{\tau,\pm}(F)\left(\tfrac{lf^2\xi}{\sqrt{p}}\right)|.\]
Since $\tilde{F}(u)$ is holomorphic everywhere except for a simple pole at $u=0$, by proposition \ref{lemasymp4}  (with $k=1$) $|\mathcal{A}_{\theta_{\infty,1}^{pos},0}^{\tau,\pm}(F)\left(\tfrac{lf^2\xi}{\sqrt{p}}\right)|=O(1)$. The claim follows.
\item \eqref{prop3.15ii}. Since $\theta_{\infty,2}^{pos}\in C_c^{\infty}(\mathbb{R})$ (cf. the paragraph before \eqref{thetajnot}) we can use the estimate in \eqref{apcor1516} of corollary \ref{corasymp15} with $a=0$,$h_1=\theta_{\infty,2}^{pos}$, $\Phi=F$. Noting that $a\equiv 0\bmod 4$, which implies that the main term vanishes, the estimate follows from
\[\eqref{prop3.15ii}\ll \left(\tfrac{lf^2}{\sqrt{p}}\right)^2\left(1+\tfrac{1}{\xi^2}\right).\]
\item \eqref{prop3.15iii} and \eqref{prop3.15iv}. The same proof in the case of \eqref{prop3.15i} applies with $F(u)$ changed to $H_1(u)$ (note that it still is holomorphic in $\Re(u)>-1$ with a simple pole at $u=0$ so we can apply proposition \ref{lemasymp4} with $k=1$.). The only difference is for \eqref{prop3.15iii} we need to take $a=0$ (which reduces the exponents in \eqref{prop3.15i} by $\tfrac12$) and for \eqref{prop3.15iv} we need to take $a=-1$ (which reduces the exponents in \eqref{prop3.15i} by $1$).
\item $\eqref{prop3.15vi}$. The argument is exactly the same as in the case of \eqref{prop3.15iii} and \eqref{prop3.15iv}. The only difference is that one needs to take into account the double pole of $H_0(u)$ at $u=0$ when applying proposition \ref{lemasymp4} (this time one needs to take $k=2$ which brings the $\log$-factor). 

\end{itemize}

\end{proof}

\begin{cor}\label{corregion2}Let $l,f,\xi\in\mathbb{Z}\backslash\{0\}$ such that $\tfrac{lf^2\xi}{\sqrt{p}}\ll 1$. Then,\begin{align*}
\int_{\mathbb{R}}\theta_{\infty}^{pos}(x)F\left(\tfrac{lf^2(4p)^{-1/2}}{\sqrt{|x^2-1|}}\right)e\left(\tfrac{-x\xi\sqrt{p}}{2lf^2}\right)dx&\ll\left(\tfrac{lf^2}{\sqrt{p}}\right)^{\frac32}\tfrac{1}{\sqrt{\xi}}\\
\tfrac{lf^2}{\sqrt{p}}\int_{|x|<1}\tfrac{\theta_{\infty}^{pos}(x)}{\sqrt{x^2-1}}H_1\left(\tfrac{lf^2(4p)^{-1/2}}{\sqrt{x^2-1}}\right)e\left(\tfrac{-x\xi\sqrt{p}}{2lf^2}\right)dx&\ll\left(\tfrac{lf^2}{\sqrt{p}}\right)^{\frac32}\tfrac{1}{\sqrt{\xi}}\\
\tfrac{lf^2}{\sqrt{p}}\int_{|x|>1}\tfrac{\theta_{\infty}^{pos}(x)}{\sqrt{x^2-1}}H_0\left(\tfrac{lf^2(4p)^{-1/2}}{\sqrt{x^2-1}}\right)e\left(\tfrac{-x\xi\sqrt{p}}{2lf^2}\right)dx&\ll\left(\tfrac{lf^2}{\sqrt{p}}\right)^{\frac32}\tfrac{\log\left(\frac{lf^2\xi}{\sqrt{p}}\right)}{\sqrt{\xi}},
\end{align*}
where the implied constants depend only on $\theta_{\infty}^{pos}$, and $F$.
\end{cor}

\begin{proof} Recall from \eqref{thetajnot2} that $\theta_{\infty}^{pos}=\theta_{\infty,1}^{pos}+\theta_{\infty,2}^{pos}$ and $\theta_{\infty,1}^{pos}(x)$ is supported in $|x|\leq1$. Then the first line follows from the combination of \eqref{prop3.15i} and \eqref{prop3.15ii}. We just need to note that since we are assuming that $\frac{lf^2\xi}{\sqrt{p}}\ll1$,
\[\tfrac{lf^2}{\sqrt{p}}\ll \tfrac{1}{\sqrt{\xi}} \hspace{0.3in}\Rightarrow \hspace{0.3in}\left(\tfrac{lf^2}{\sqrt{p}}\right)^{2}\ll \left(\tfrac{lf^2}{\sqrt{p}}\right)^{\frac32}\tfrac{1}{\sqrt{\xi}}.\]
The second line directly follows from the combination of \eqref{prop3.15ii} and \eqref{prop3.15iv}, and the last line follows from \eqref{prop3.15vi}.

\end{proof}

The following corollary puts together the results above and reduces the estimation of $\Sigma(\xi\neq0)$ to bounding the sums below.

\begin{cor}\label{finalanalyticest}For every $N>0$,
\begin{equation}\label{finalanalyticesteqn}
\Sigma(\xi\neq0)\ll \sqrt{p}\sum_{\substack{l,f,|\xi|\geq1\\ \frac{lf^2\xi}{\sqrt{p}}\gg 1}}^{\infty}\tfrac{Kl_{l,f}(\xi,p)}{\sqrt{l}}\tfrac{1}{(lf^2)^{\frac32}}\left(\tfrac{\sqrt{p}}{lf^2\xi}\right)^{N}\tfrac{1}{\xi^{2}}+\tfrac{1}{p^{\frac14}}\sum_{\substack{l,f,\xi\\ \frac{lf^2\xi}{\sqrt{p}}\ll1}}^{\infty}\tfrac{Kl_{l,f}(\xi,p)}{\sqrt{l\xi}}\left\{1+\log\left(\tfrac{lf^2\xi}{\sqrt{p}}\right)\right\},
\end{equation}
where the implied constants depend only on $\theta_{\infty}^{pos},\theta_{\infty}^{neg},F$, and $N$.
\end{cor}

\begin{proof}Recall from theorem \ref{recallthm} that $\Sigma(\xi\neq0)$ is the sum of five terms. The first three of these terms have integrals involving $\theta_{\infty}^{pos}$ and last two have integrals involving $\theta_{\infty}^{neg}$. Corollaries \ref{corregion1} and \ref{corregion2} immediately imply that each of the first three lines are bounded by the right hand side of \eqref{finalanalyticesteqn}. 

For the last two terms note that if we can show that the integrals involving $\theta_{\infty}^{neg}$ satisfy the same bounds as the corresponding integrals of $\theta_{\infty}^{pos}$ in corollaries \ref{corregion1} and \ref{corregion2} then we are done. To do so, we use proposition \ref{smoothpropasymp} with suitable choices of $M$ and $N_1$. 

For the region $\tfrac{lf^2\xi}{\sqrt{p}}\gg1$ we choose $M=2+N$ and $N_1=2+2N$. This gives,
\begin{align*}
\int_{\mathbb{R}}\theta_{\infty}^{neg}(x)F\left(\tfrac{lf^2(4p)^{-1/2}}{\sqrt{x^2+1}}\right)e\left(\tfrac{-x\xi\sqrt{p}}{2lf^2}\right)dx&\ll \left(\tfrac{\sqrt{p}}{lf^2\xi}\right)^N\tfrac{1}{\xi^2}\\
\int_{\mathbb{R}}\tfrac{\theta_{\infty}^{neg}(x)}{\sqrt{x^2+1}}H_0\left(\tfrac{lf^2(4p)^{-1/2}}{\sqrt{x^2+1}}\right)e\left(\tfrac{-x\xi\sqrt{p}}{2lf^2}\right)dx&\ll\left(\tfrac{\sqrt{p}}{lf^2\xi}\right)^{N}\tfrac{1}{\xi^2}.
\end{align*}
For the region $\tfrac{lf^2\xi}{\sqrt{p}}\ll1$ we choose $M=2$ and $N_1=0$. This gives,
\begin{align*}
\int_{\mathbb{R}}\theta_{\infty}^{neg}(x)F\left(\tfrac{lf^2(4p)^{-1/2}}{\sqrt{x^2+1}}\right)e\left(\tfrac{-x\xi\sqrt{p}}{2lf^2}\right)dx&\ll \left(\tfrac{lf^2}{\sqrt{p}}\right)^2\tfrac{1}{\xi^2}\\
\int_{\mathbb{R}}\tfrac{\theta_{\infty}^{neg}(x)}{\sqrt{x^2+1}}H_0\left(\tfrac{lf^2(4p)^{-1/2}}{\sqrt{x^2+1}}\right)e\left(\tfrac{-x\xi\sqrt{p}}{2lf^2}\right)dx&\ll\left(\tfrac{lf^2}{\sqrt{p}}\right)^2\tfrac{1}{\xi^2}.
\end{align*}
Finally note that $\tfrac{lf^2\xi}{\sqrt{p}}\ll1$ implies $\left(\tfrac{lf^2}{\sqrt{p}}\right)^2\tfrac{1}{\xi^2}\ll \left(\tfrac{lf^2}{\sqrt{p}}\right)^{\frac32}\tfrac{1}{\sqrt{\xi}}$. This finishes the proof.
\end{proof}

\subsubsection{Estimating $\Sigma(\xi\neq0)$} We are now ready to finish the proof of theorem \ref{finalanalyticest} by estimating the sums in corollary \ref{finalanalyticest}. The only remaining part is to bound the character sums, $Kl_{l,f}(\xi,p)$, which the following lemma handles.

\begin{lemma}\label{lem3.18} Let $p$ be an odd prime, $l,f,\xi\in\mathbb{Z}\backslash \{0\}$, and $Kl_{l,f}(\xi,p)$ be as in theorem \ref{recallthm}. Then, for any function, $F(l,f,\xi)$, we have,
\begin{align*}
\sum_{\substack{l,f,\xi\\ \frac{lf^2\xi}{\sqrt{p}}\ll1}}F(l,f,\xi)Kl_{l,f}(\xi,p)&\ll \sum_{\substack{\kappa,l,f,\xi\\ \frac{\kappa^2 lf^2\xi}{\sqrt{p}}\ll1}}\log(\kappa lf^2)\kappa\sqrt{l}F(l\kappa,f,\xi\kappa).
\end{align*}
Moreover, the same bound also holds for the sum with boundary $\frac{lf^2\xi}{\sqrt{p}}\gg1$.
\end{lemma}

\begin{proof} By corollary \ref{charsumcor3},
\[Kl_{l,f}(\xi,p)\ll\delta(p;f^2)\begin{cases}\log(lf^2)\sqrt{l\gcd(p,f^2)}\sqrt{\gcd\left(\tfrac{\xi}{\sqrt{\gcd(p,f^2)}},l\right)}&\tfrac{l\sqrt{\gcd(p,f^2)}}{rad(l)}\mid \xi\\ 0&otherwise\end{cases}.\]
Because of the factor, $\delta(p,f^2)$, the the sum does not vanish only if $\gcd(p,f^2)=1$ and $p$ is a square modulo $f$. Since $p$ is a prime this forces $\gcd(p,f)=1$ and implies
\[Kl_{l,f}(\xi,p)\ll\log(lf^2)\sqrt{l}\sqrt{\gcd\left(\xi,l\right)}\tag{$*$}\label{lem3.18*}.\]
Let $\kappa=\gcd(\xi,l)$. Then, 
\begin{align*}
\sum_{\substack{l,f,\xi\\ \frac{lf^2\xi}{\sqrt{p}}\ll1}}F(l,f,\xi)Kl_{l,f}(\xi,p)&\ll \sum_{\substack{\kappa,l,f,\xi\\ \gcd(l,\xi)=1\\ \frac{\kappa^2lf^2\xi}{\sqrt{p}}\ll1}}\log\left(\kappa lf^2\right)\kappa\sqrt{l}F\left(l\kappa,f,\xi\kappa\right).
\end{align*}
This proves the inequality of the lemma. Since the argument does not use anything about the particular boundary of summation the statement about the region $\frac{lf^2\xi}{\sqrt{p}}\gg1$ is obvious.
\end{proof}

\begin{prop}\label{propfirstest}
\begin{equation}\label{first}
\sqrt{p}\sum_{l,f=1}^{\infty}\tfrac{1}{(lf^2)^{\frac32}}\sum_{\substack{\xi\in\mathbb{Z}\backslash\{0\}\\ \frac{lf^2\xi}{\sqrt{p}}\gg 1}}\tfrac{Kl_{l,f}(\xi,p)}{\sqrt{l}}\left(\tfrac{\sqrt{p}}{lf^2\xi}\right)^{N}\tfrac{1}{\xi^{2}}\ll p^{\frac14}.
\end{equation}
\end{prop}

\begin{proof}
By lemma \ref{lem3.18}, 
\begin{equation*}\tag{$\ref{first}'$}\label{braketheregion}
\eqref{first}\ll\sqrt{p}\sum_{\substack{\kappa,l,f,\xi\\ \frac{\kappa^2lf^2\xi}{\sqrt{p}}\gg1}}\tfrac{\log(\kappa lf^2)}{(\kappa lf^2)^{\frac32}}\tfrac{\sqrt{\kappa}}{(\kappa\xi)^{2}}\left(\tfrac{\sqrt{p}}{\kappa^2 lf^2\xi}\right)^{N}
\end{equation*}
We break the region of summation $\frac{\kappa^2lf^2\xi}{\sqrt{p}}\gg1$ into sub-regions,
\begin{align*}
\eqref{braketheregion}&= \sum_{\substack{\kappa^2\gg \sqrt{p}\\ l,f,|\xi|\geq1}}\cdots+\sum_{\substack{\kappa^2\ll \sqrt{p}\\ lf^2\xi\gg \frac{\sqrt{p}}{\kappa^2}}}\cdots= \sum_{\substack{\kappa^2\gg \sqrt{p}\\ l,f,|\xi|\geq1}}+\sum_{\substack{\kappa\ll \sqrt{p}\\ f^2\gg \frac{\sqrt{p}}{\kappa^2}\\ l,\xi\geq1 }}+\sum_{\substack{\kappa^2\ll \sqrt{p}\\ f^2\ll \frac{\sqrt{p}}{\kappa^2}\\l\xi\gg \frac{\sqrt{p}}{\kappa^2f^2}}}\\
&= \sum_{\substack{\kappa^2\gg \sqrt{p}\\ l,f,|\xi|\geq1}}+\sum_{\substack{\kappa\ll \sqrt{p}\\ f^2\gg \frac{\sqrt{p}}{\kappa^2}\\ l,|\xi|\geq1 }}+\sum_{\substack{\kappa^2\ll \sqrt{p}\\ f^2\ll \frac{\sqrt{p}}{\kappa^2}\\ l\gg \frac{\sqrt{p}}{\kappa^2f^2}\\  |\xi|\geq1}}+\sum_{\substack{\kappa^2\ll \sqrt{p}\\ f^2\ll \frac{\sqrt{p}}{\kappa^2}\\ l\ll \frac{\sqrt{p}}{\kappa^2f^2}\\  \xi\gg\frac{\sqrt{p}}{\kappa^2lf^2}}}\\
&=\eqref{braketheregion}_{a}+\eqref{braketheregion}_b+\eqref{braketheregion}_c+\eqref{braketheregion}_d.
\end{align*}
We now bound $\eqref{braketheregion}_a$ to $\eqref{braketheregion}_d$.
\begin{itemize}
\item $\eqref{braketheregion}_a$.
\begin{align*}
\eqref{braketheregion}_a&\ll p^{\frac{N+1}{2}}\sum_{\kappa^2\gg \sqrt{p}}\tfrac{1}{\kappa^{3+2N}}\sum_{l,f,|\xi|\geq1}\tfrac{\log\left(\frac{\kappa^2lf^2\xi}{\sqrt{p}}\right)}{l^{\frac32+N}f^{3+2N}\xi^{2+N}}\ll \log(p).
\end{align*}
\item $\eqref{braketheregion}_b$.
\begin{align*}
\eqref{braketheregion}_b&\ll p^{\frac{N+1}{2}}\sum_{\kappa^2\ll \sqrt{p}}\tfrac{1}{\kappa^{3+2N}}\sum_{f^2\gg \frac{\sqrt{p}}{\kappa^2}}\tfrac{1}{f^{3+2N}}\sum_{l,|\xi|\geq1}\tfrac{\log\left(\frac{\kappa^2lf^2\xi}{\sqrt{p}}\right)}{l^{\frac32+N}\xi^{2+N}}\\
&\ll p^{\frac{N+1}{2}}\sum_{\kappa^2\ll \sqrt{p}}\tfrac{1}{\kappa^{3+2N}} \left(\tfrac{\kappa^2}{\sqrt{p}}\right)^{2+2N}\ll \log(p).
\end{align*}
\item $\eqref{braketheregion}_c$.
\begin{align*}
\eqref{braketheregion}_c&\ll p^{\frac{N+1}{2}}\sum_{\kappa^2\ll \sqrt{p}}\tfrac{1}{\kappa^{3+2N}}\sum_{f^2\ll \frac{\sqrt{p}}{\kappa^2}}\tfrac{1}{f^{3+2N}}\sum_{l\gg \frac{\sqrt{p}}{\kappa^2f^2}} \tfrac{1}{l^{\frac32+N}}\sum_{|\xi|\geq1}\tfrac{\log\left(\frac{\kappa^2lf^2\xi}{\sqrt{p}}\right)}{\xi^{2+N}}\\
&\ll p^{\frac{N+1}{2}}\sum_{\kappa^2\ll \sqrt{p}}\tfrac{1}{\kappa^{3+2N}}\sum_{f^2\ll \frac{\sqrt{p}}{\kappa^2}}\tfrac{1}{f^{3+2N}}\sum_{l\gg \frac{\sqrt{p}}{\kappa^2f^2}} \tfrac{\log\left(\frac{\kappa^2lf^2}{\sqrt{p}}\right)}{l^{\frac32+N}}\\
&\ll p^{\frac{N+1}{2}}\sum_{\kappa^2\ll \sqrt{p}}\tfrac{1}{\kappa^{3+2N}}\sum_{f^2\ll \frac{\sqrt{p}}{\kappa^2}}\tfrac{1}{f^{3+2N}}\left(\tfrac{\kappa^2f^2}{\sqrt{p}}\right)^{\frac12+N}\ll p^{\frac14}.
\end{align*}
\item $\eqref{braketheregion}_d$.
\begin{align*}
\eqref{braketheregion}_d&\ll p^{\frac{N+1}{2}}\sum_{\kappa^2\ll \sqrt{p}}\tfrac{1}{\kappa^{3+2N}}\sum_{f^2\ll \frac{\sqrt{p}}{\kappa^2}}\tfrac{1}{f^{3+2N}}\sum_{l\ll \frac{\sqrt{p}}{\kappa^2f^2}} \tfrac{1}{l^{\frac32+N}}\sum_{\xi\gg \frac{\sqrt{p}}{\kappa^2lf^2}}\tfrac{\log\left(\frac{\kappa^2lf^2\xi}{\sqrt{p}}\right)}{\xi^{2+N}}\\
&\ll p^{\frac{N+1}{2}}\sum_{\kappa^2\ll \sqrt{p}}\tfrac{1}{\kappa^{3+2N}}\sum_{f^2\ll \frac{\sqrt{p}}{\kappa^2}}\tfrac{1}{f^{3+2N}}\sum_{l\ll \frac{\sqrt{p}}{\kappa^2f^2}} \tfrac{1}{l^{\frac32+N}}\left(\tfrac{\kappa^2lf^2}{\sqrt{p}}\right)^{1+N}\\
&\ll \sum_{\kappa^2\ll \sqrt{p}}\tfrac{1}{\kappa}\sum_{f^2\ll \frac{\sqrt{p}}{\kappa^2}}\tfrac{1}{f}\left(1+\left(\tfrac{\sqrt{p}}{\kappa^2f^2}\right)^{\frac12}\right)\\
&\ll \sum_{\kappa^2\ll \sqrt{p}}\tfrac{1}{\kappa}\left(1+\log\left(\tfrac{p^{\frac14}}{\kappa}\right)+\tfrac{p^{\frac14}}{\kappa}\right)\ll p^{\frac14}.
\end{align*}
\end{itemize}
Combining the above estimates for $\eqref{braketheregion}_a$ to $\eqref{braketheregion}_d$ finishes the proof.
\end{proof}

\begin{prop}\label{propsecondest}
\begin{equation}\label{second}
\tfrac{1}{p^{\frac14}}\sum_{\substack{l,f,\xi\\ \frac{lf^2\xi}{\sqrt{p}}\ll1}}^{\infty}\tfrac{Kl_{l,f}(\xi,p)}{\sqrt{l\xi}}\left\{1+\log\left(\tfrac{lf^2\xi}{\sqrt{p}}\right)\right\}\ll p^{\frac14}.
\end{equation}
\end{prop}

\begin{proof}Using lemma \ref{lem3.18},
\begin{align*}
\eqref{second}&\ll \tfrac{1}{p^{\frac14}}\sum_{\substack{\kappa,l,f,\xi\\ \frac{\kappa^2 lf^2\xi}{\sqrt{p}}\ll1}}\tfrac{\log(\kappa lf^2)}{\sqrt{\xi}}\left\{1+\log\left(\tfrac{\kappa^2 lf^2\xi}{\sqrt{p}}\right)\right\}
\end{align*}
We then bound each sum one-by-one.
\begin{align*}
&=\tfrac{1}{p^{\frac14}}\sum_{\kappa^2\ll \sqrt{p}}\sum_{f^2\ll \frac{\sqrt{p}}{\kappa^2}}\sum_{l\ll \frac{\sqrt{p}}{\kappa^2f^2}}\sum_{\xi\ll \frac{\sqrt{p}}{\kappa^2lf^2}}\tfrac{\log(\kappa lf^2)}{\sqrt{\xi}}\left\{1+\log\left(\tfrac{\kappa^2lf^2\xi}{\sqrt{p}}\right)\right\}\\
&\ll\tfrac{1}{p^{\frac14}}\sum_{\kappa^2\ll \sqrt{p}}\sum_{f^2\ll \frac{\sqrt{p}}{\kappa^2}}\sum_{l\ll \frac{\sqrt{p}}{\kappa^2f^2}}\log(\kappa lf^2)\left\{1+\left(\tfrac{\sqrt{p}}{\kappa^2lf^2}\right)^{\frac12}\right\}\\
&\ll\tfrac{1}{p^{\frac14}}\sum_{\kappa^2\ll \sqrt{p}}\sum_{f^2\ll \frac{\sqrt{p}}{\kappa^2}}\left\{1+\left(\tfrac{\sqrt{p}}{\kappa^2f^2}\right)^{\frac12}\log\left(\tfrac{\sqrt{p}}{\kappa^2f^2}\right)\right\}\\
&\ll\tfrac{1}{p^{\frac14}}\sum_{\kappa^2\ll \sqrt{p}}\left\{1+\left(\tfrac{\sqrt{p}}{\kappa^2}\right)\log\left(\tfrac{\sqrt{p}}{\kappa^2}\right)\right\}\ll p^{\frac14}.
\end{align*}

\end{proof}

\appendix
\section{Asymptotic expansions of
Fourier transforms}\label{apa}

In this appendix we will derive the asymptotic expansions that are used extensively throughout the paper. The main results of the appendix are theorems \ref{thmasymp}, \ref{thmasymp2} and corollary \ref{corasymp15}.

\subsection{Preliminaries}

We begin by going over some basics that will be used in studying the asymptotic behavior of the Fourier transforms.

\subsubsection{Cut-off functions}

Throughout the proofs we will use classical partition of unity arguments. In order to keep track of uniformity in parameters we will do this as explicitly as possible. In this section we give details about the our cut-off functions and their relevant properties.

\begin{definition}[$\mathcal{G}_0$] Let $\mathcal{G}_0$ denote the set of functions 
$\phi_0:[0,\infty)\rightarrow \mathbb{R}$ such that:
\begin{itemize}
\item $supp(\phi_0)= [0,1]$ and $\phi_0(0)=1$,
\item $ \lim_{x\rightarrow 1^-}\phi_0(x)=0$, 
\item $\phi_0$ is smooth, positive and monotonically
decreasing,
\item All the right derivatives of $\phi_0$ at $0$ and the left
derivatives at $1$ vanish,
\end{itemize}
\end{definition}
(To see that $\mathcal{G}_0\neq \varnothing$ one can consider,
for example,
$\phi_0(x)=\iota\int_0^{1-x}e^{-\frac{1}{1-y}-\frac{1}{y}}dy$,
extended by $0$ for $x>1$, and $\iota$ is normalized so that $
\phi_0(0)=1$.) 

\begin{definition}[$\mathcal{G}$]\label{apdef2}Let $\mathcal{G}$ denote the set of functions $\phi(x):[0,\infty]\rightarrow \mathbb{R}$ defined by,
\begin{equation}\label{phi}
\phi(x):=\begin{cases} 1&\text{if $0\leq x\leq 1$,}\\
\phi_0(x-1)&\text{if $x\geq1$,}\end{cases}
\end{equation}
for some $\phi_0\in\mathcal{G}_0$.
\end{definition}

\begin{definition}\label{sing27}Given any $\phi\in \mathcal{G}$ and $\kappa\in\mathbb{R}_{>0}$ we define $\phi_{\kappa}(x)$ by
\begin{equation*}
\phi_{\kappa}(x):=\phi\left(\tfrac{|1-|x||}{\kappa}\right).
\end{equation*}
\end{definition}

The next few lemmas derive elementary properties of certain functions that will be used in the proofs.

\begin{lemma}\label{lemmaa1} For any $\phi\in\mathcal{A}$ the Mellin transform,
$\tilde{\phi}(s)$, is analytic everywhere excepts for a
simple pole at $s=0$ with residue $1$. Furthermore $\tilde{\phi}(s)$ decays faster than any polynomial of $|s|$. 

\end{lemma}

\begin{proof}Note that since $\phi(x)$ is compactly
supported the integral defining the Mellin transform converges for $\Re(s)>0$ and defines a holomorphic function in that region.
Integrating by parts gives,
\[\tilde{\phi}(x)=-\tfrac{1}{s}\int_{0}^{\infty}\phi'(x)x^sdx.\]
Now note that since $\phi(x)$ is constant on $[0,1]$
$\phi'(x)$ is supported in $[1,2]$ and hence the integral
above converges for every $s\in\mathbb{C}$. This gives the
analytic continuation and the simple pole at $s=0$. The residue
is easily calculated to be $\phi(0)=1$ by the fundamental
theorem of calculus. The last claim follows by repeated integration by parts.

\end{proof}

\subsubsection{Generalities about asymptotic expansions}\label{apsecasympexp}

Recall that smoothness properties of a function is reflected on the asymptotic properties of its Fourier transform. Since the functions we are interested in are, in general, not smooth (cf. \eqref{sectriv2}), in order to understand the behavior of the Fourier transforms we will need to study them around the singularities of the functions. To work locally around singularities we will be using asymptotic expansions. In this section we will describe basic definitions and properties of asymptotic expansions that will be relevant to the paper. For more on asymptotic expansions see, for instance chapter 5 of \cite{Bourbaki:2004aa}

By an \emph{asymptotic expansion} of a function, $f(x)$, defined on a domain $D$, around $\alpha\in D$, we will understand a series expansion 
\begin{equation}
f(\alpha-x)\sim |x|^{\beta}\sum_{m=0}^{\infty}c_m x^m\hspace{0.4in},or\hspace{0.4in}f(\alpha+x)\sim |x|^{\beta}\sum_{m=0}^{\infty}c_m x^m
\end{equation}
with constants $\beta,c_m\in\mathbb{C}$, such that for every $M\in\mathbb{N}$ there exists a neighborhood, $D_M\subset D$, of $\alpha$ such that for all $\alpha-x\in D_M$, $f(\alpha-x)=|x|^\beta\sum_{m=0}^Mc_mx^m+f_{M+1}(\alpha-x)$, where $f_{M+1}(\alpha-x)=O(x^{\beta+M+1})$, and the constant in the big-$O$ depends only on $\alpha, f$ and $M$. 

These expansions will be useful when describing the behavior of the function around a singular point $x=\alpha$, and for us $\alpha$ will be only $\pm1$. Note also that this expansion does not have to be a Taylor expansion (for example the exponent $\beta$ may be non-integral). 

The following statement is an immediate consequence of the definitions however we present it separately because we will be needing the exact expression for the constants (cf. corollary \ref{corasymp15}).

\begin{lemma}\label{aplema5}Let $h(x)$ have the asymptotic expansion
\[h(\pm(1-x))\sim |x|^{\beta}\sum_{m=0}^{\infty}c_m^{\pm}x^m.\]
Then, for any $\delta\in\mathbb{C}$ and $|x|<2$ we have
\[|2x-x^2|^{\delta}h(\pm(1-x))\sim |x|^{\beta+\delta}\sum_{m}d_m^{\pm}x^m\hspace{0.3in},where\hspace{0.3in} d_m^{\pm}=2^{\delta}\sum_{j+k=m}\tfrac{c_k^{\pm}}{(-2)^j}\binom{\delta}{j}.\]

\end{lemma}

\begin{proof}
Obvious.

\end{proof}

\subsubsection{An integral tranform}

In this section we introduce an integral transform that will appear when one considers the asymptotic behavior of Fourier transforms of the types we consider in the text.

For the following definition, let $a\in\mathbb{C}$ with $\Re(a)>-2$ be a constant and let $h_a(x)$ be a function with the following asymptotic expansions around $x=\pm1$,
\[h_a(\pm(1-x))\sim |x|^\frac{a}{2}\sum_{n=0}^\infty c_m^{\pm}x^m,\]
where $c_m^{\pm}\in\mathbb{C}$ are constants.

\begin{definition}\label{inttrans} Let $\Phi\in\mathcal{S}(\mathbb{R})$,  $m\in\mathbb{N}$, and $\tau\in\mathbb{C}$ with $\Re(\tau)>0$. Define $ \mathcal{A}_{h_a,m}^{\tau,\pm}(\Phi)(x)$ by
\begin{align*}
 \mathcal{A}_{h_a,m}^{\tau,\pm}(\Phi)(x)&:=\tfrac{1}{2\pi i }\int_{(\tau)}\tilde{\Phi}(u)c_{m}^{\pm}\left(\tfrac{u}{2}\right)\Gamma\left(m+1+\tfrac{a+u}{2}\right)x^{-\frac{u}{2}}du,
\end{align*}
where $\Gamma$ denotes the usual gamma function and,
\begin{align*}
c_{m}^{\pm}\left(\tfrac{u}{2}\right)&:=\left(\tfrac{i}{2\pi }\right)^{1+m+\frac {u+a}{2}}2^{\frac{u}{2}}\sum_{\substack{j+k=m\\ j,k\geq0}}\tfrac{{c_k^{\pm}}}{(-2)^j}\binom{\frac u2}{j}.
\end{align*}
Note that $c_m^{\pm}\left(\frac u2\right)$ is a holomorphic function of $u$.

\end{definition}

The proposition below establishes the decay properties of $\mathcal{A}_{h_a,m}^{\tau,\pm}(\Phi)(x)$ for large $x$.

\begin{prop}\label{lemasymp3}Let $h_a(x)$ be as above, $m\in\mathbb{N}$, and $\Phi\in \mathcal{S}(\mathbb{R})$. Then for any $\tau>0$ we have
\[\mathcal{A}_{h_a,m}^{\tau,\pm}(\Phi)(x)=O(x^{-\tau}),\]
where the implied constant depends only on $\Phi,h_a,m$ and $\tau$.
\end{prop}

\begin{proof} We just need to observe that since $\Phi\in\mathcal{S}(\mathbb{R})$ its Mellin transform is holomorphic in the half plane $\Re(u)>0$. With the rest of the functions in the integral transform also being holomorphic in the same half-plane, we are free to move the $u$-contour anywhere in $\Re(u)=2\tau>0$.
\end{proof}

For small values of $x$ we have the following estimate.
\begin{prop}\label{lemasymp4}Let $h_a(x)$ be as above, $m,k\in\mathbb{N}$ such that $2m+2+a>0$. Let $\Phi\in \mathcal{S}(\mathbb{R})$ such that $\tilde{\Phi}(u)$ is holomorphic in $\Re(u)\geq-\epsilon$ for some $\epsilon>0$ and has a pole of order $k$ at $u=0$. Then,
\[\mathcal{A}_{h_a,m}^{\tau,\pm}(\Phi)(x)=O(\log^{k-1}(x)),\]
where the implied constant depends only on $\Phi,h_a,m$ and $\tau$.

\end{prop}

\begin{proof}This follows from the Cauchy integral formula. Pushing the $u$-contour to $\Re(u)=-\min\{\epsilon,m+1+\frac{a}{2}\}$ picks up the contribution at $u=0$ (Note that by the assumption $2m+2+a>0$, $\Gamma\left(m+1+\frac{a+u}{2}\right)$ does not contribute to the pole.). Let $\tilde{\Phi}(u)=\sum_{i=-k}^{\infty}\alpha_iu^i$ be the Laurent expansion of $\tilde{\Phi}(u)$ around $u=0$. Then, by the Cauchy integral formula this contribution is 
\[\sum_{j=1}^{k}\tfrac{\alpha_{-j}}{(j-1)!}\tfrac{d^{j-1}}{du^{j-1}}\left.\tfrac{c_m^{\pm}\left(\tfrac u2\right)\Gamma\left(m+1+\frac{a+u}{2}\right)}{\Gamma\left(\frac{1-u}{2}\right)}x^{-\frac u2}\right|_{u=0}=O(\log^{k-1}(x)).\]

\end{proof}

\subsection{Asymptotic properties of Fourier transforms}

In \S\ref{smoothfourier} we study the decay properties of Fourier transforms of smooth functions of a certain type which depend on a parameter $C$. In \S\ref{singularfourier} we discuss the same problem for functions with prescribed singularities. Since the functions considered are not smooth anymore the Fourier transforms decay slower and oscillate. We give explicit asymptotic expansions describing the decay rate and oscillation frequency in terms of the singularities of the function and the relevant parameters. We emphasize that the main point of both sections (cf. corollaries \ref{cortoaplem4}, \ref{corasymp15} and theorems \ref{thmasymp}, \ref{thmasymp2}) is the independence of the implied constants of the parameter $C$ and $D$.

\subsubsection{Fourier transforms of smooth functions}\label{smoothfourier}

\begin{lemma}\label{aplem4}Let $\frac12>\kappa>0$ be a constant, and $\phi_{\kappa}$ be as in definition \ref{sing27}. Let $\Phi\in \mathcal{S}(\mathbb{R})$, $h\in C_c(\mathbb{R})$ that is smooth on the support of $1-\phi_k$, $a\in\mathbb{C}$, and $C,D\in\mathbb{R}\backslash \{0\}$. Then for any $N\in\mathbb{R}_{>0},M\in \mathbb{Z}_{\geq0}$,

\[\int_{\mathbb{R}}(1-\phi_\kappa(x))h(x)|x^2\pm1|^{\frac{a}{2}}\Phi\left(\tfrac{C}{\sqrt{|x^2\pm1|}}\right)e(xD)dx=O\left( C^{-N}D^{-M}\right),\]
where the implied constant depends only on $
h,\Phi,a,\phi_{\kappa},M$ and $N$.

\end{lemma}

\begin{proof} Let us denote the integral by $I(C,D)$. First, note that since $1-\phi_{\kappa}$ is identically $0$ for $|x\pm1|<\kappa$ the functions $|x^2\pm1|^{\frac a2}$ are all smooth in the region of integration. Furthermore, since the support of $1-\phi_{\kappa}(x)$ is compact, so is the region of integration. Then, for any $\tau\in\mathbb{R}_{>0}$ by Mellin inversion we can write
\[I(C,D)=\tfrac{1}{2\pi i}\int_{(\tau)}\tfrac{\tilde{\Phi}(u)}{C^{u}}\int_{\mathbb{R}}(1-\phi_\kappa(x))h(x)|x^2\pm1|^{\frac{u+a}{2}}e(xD)dxdu.\]
Note that the interchange of integrals is justified since $\tilde{\Phi}(u)$ decays faster than any polynomial in $|u|$, the $x$-integral is over a compact region, and the integrands are smooth in the region of integration hence the double integral converges. Then,  integration by parts $M$-times in the $x$-integral gives,
\[I(C,D)=\tfrac{(2\pi iD)^{-M}}{2\pi i}\int_{(\tau)}\tfrac{\tilde{\Phi}(u)}{C^{u}}\int_{\mathbb{R}}\tfrac{d^M}{dx^M}\left\{(1-\phi_\kappa(x))h(x)|x^2\pm1|^{\frac{u+a}{2}}\right\}e(xD)dxdu.\]
(Boundary terms vanish because of $(1-\phi_{\kappa}(x))$.) We now interchange the $x$ and $u$ integrals once again and push the $u$-contour to $\Re(u)=N$ and get,
\[I(C,D)=\tfrac{(2\pi iD)^{-M}}{2\pi i}\int_{\mathbb{R}}\int_{(N)}\tfrac{\tilde{\Phi}(u)}{C^{u}}\tfrac{d^M}{dx^M}\left\{(1-\phi_\kappa(x))h(x)|x^2\pm1|^{\frac{u+a}{2}}\right\}e(xD)dudx.\]
Since the integrand of the $x$-integral is smooth in the domain of integration by assumption we get that the resulting integral is bounded. The result follows.
\end{proof}

\begin{cor}\label{cortoaplem4}Let $\Phi\in \mathcal{S}(\mathbb{R})$, $h\in C_c^{\infty}(\mathbb{R})$, $a\in\mathbb{C}$, and $C,D\in\mathbb{R}\backslash \{0\}$. Then for any $M,N\in \mathbb{Z}_{\geq0}$,

\[\int_{\mathbb{R}}h(x)(x^2+1)^{\frac{a}{2}}\Phi\left(\tfrac{C}{\sqrt{x^2+1}}\right)e(xD)dx=O\left( C^{-N}D^{-M}\right),\]
where the implied constant depends only on $
h,\Phi,\delta,M$ and $N$.

\end{cor}

\begin{proof}The proof of lemma \ref{aplem4} goes through verbatim.

\end{proof}

\subsubsection{Fourier transforms of functions with prescribed singularities}\label{singularfourier}

\paragraph{Technical lemmas.}

In this section we present the collection of technical lemmas that are used to obtain the asymptotic expansions of the Fourier transforms. These lemmas can be skipped for the first read and the reader can go straight to the end results of theorems \ref{thmasymp} and \ref{thmasymp2}. We note, however, that the analysis given in lemma \ref{lemasymp2} is fundamental for the paper, and although the proofs are technical the arguments are quite elementary. Once again, we emphasize that the most important point to keep in mind, especially about lemma \ref{lemasymp2}, is that all the implied constants are independent of the parameters $C$ and $D$. This is a central issue since in the applications (cf. \S\ref{secasympfourier}) we will have $C$ and $D$ depending on further parameters that will be summed over, and our aim is to use the asymptotic expansions to study these sums.

\begin{lemma}\label{lemasymp} Let $A, B\in\mathbb{C}$ with $\Re(A)>-1$ (for convergence), $\iota\in\{0,1\}$, $\tfrac12>\kappa>0$, and $Z\in\mathbb{R}\backslash\{0\}$. Define $I_{\iota,\kappa}(A,B,Z)$ by,
\[I_{{\iota},\kappa}(A,B,Z):=\int_0^{\frac12}x^{A}(2+(-1)^{\iota} x)^{B}\phi\left(\tfrac{x}{\kappa}\right)e(xZ)dx.\]

Then, for any $\sigma_1,\sigma_2\in\mathbb{R}_{>0}$ we have,
\begin{align*}
I_{\iota,\kappa}(A,B,Z)&=\left(\tfrac{i}{2\pi Z}\right)^{A+1}\int_0^{\infty}x^{A}\left(2+(-1)^{\iota} \tfrac{ix}{2\pi Z}\right)^{B}e^{-x}dx\\
&+\tfrac{1}{2\pi i}\int_{(-\sigma_1)}\tilde{\phi}(s)\kappa^s\left\{\left(\tfrac{i}{2\pi Z}\right)^{A-s+1}\int_0^{\infty}x^{A-s}\left(2+(-1)^{\iota} \tfrac{ix}{2\pi Z}\right)^{B}e^{-x}dx\right\}ds.\\
&-\tfrac{1}{2\pi i}\int_{(\sigma_2)}\tilde{\phi}(s)\kappa^s\int_1^{1+\frac{|Z|}{Z}i\infty}x^{A-s}(2+(-1)^{\iota} x)^{B}e(xZ)dxds.
\end{align*}

\end{lemma}

\begin{proof}

We will deform the contour in a suitable half-plane while keeping track of the residues. Below, we will give the proof in detail for the case $Z>0$. The only difference for $Z<0$ is to deform the contour in the lower half-plane rather than the upper half plane. We will point to appropriate modifications in the proof as we move on.

Let $\sigma_0>0$ be such that $\Re(A+\sigma_0)>-1$ (note that such $\sigma_0$ exists because $\Re(A)>-1$) . By Mellin inversion the integral is
\[\tfrac{1}{2\pi i}\int_{(\sigma_0)}\tilde{\phi}(s)\kappa^s\int_0^{\frac12}x^{A-s}(2+(-1)^{\iota} x)^{B}e(xZ)dxds.\tag{$*$}\label{int1}\]
We note that the interchange of the order of integration is justified by the absolute convergence of the double integral for $\Re(A-s)>-1$. Because of the decay of the exponential it is straightforward to see that the contour in the inner integral can be deformed to
\begin{align*}
\int_0^{\frac12}x^{A-s}(2+(-1)^{\iota} x)^{B}e(xZ)dx&=\int_0^{i\infty}\cdots dx+\int_{\frac12+i\infty}^{\frac12}\cdots dx\\
&\hspace{-1in}=\left(\tfrac{i}{2\pi Z}\right)^{A-s+1}\int_0^{\infty}x^{A-s}\left(2+(-1)^{\iota} \tfrac{ix}{2\pi Z}\right)^{B}e^{-x}dx-\int_{\frac12}^{\frac{1}{2}+i\infty}x^{A-s}(2+(-1)^{\iota} x)^{B}e(xZ)dx,
\end{align*}
(If $Z<0$ we deform the contour to $\int_0^{-i\infty}+\int_{\frac12-i\infty}^{\frac12}$). Where, we have used $x\mapsto \frac{ix}{2\pi Z}$ in the first integral to get the second equality. We substitute this into \eqref{int1} and analyze each term separately. Considering the second term we get
\[\tfrac{1}{2\pi i}\int_{(\sigma_0)}\tilde{\phi}(s)\kappa^s\int_{\frac12}^{1+i\infty}x^{A-s}(2+(-1)^{\iota} x)^{B}e(xZ)dxds.\]
Note that in this integral we can shift the $s$-contour to any $\sigma_2>0$ since the $x$-integral converges for every $s$. This gives the last term of the lemma. 

 We then consider the first term. Shifting the $s$-contour to $\Re(s)=-\sigma_1$ (Recall that by lemma \ref{lemmaa1} $\tilde{\phi}(s)$ has a simple pole with residue $1$ at $s=0$ and is holomorphic everywhere else.) we get,
\begin{align*}
\tfrac{1}{2\pi i}\int_{(\sigma_0)}\tilde{\phi}(s)\kappa^s\left(\tfrac{i}{2\pi Z}\right)^{A-s+1}&\int_0^{\infty}x^{A-s}\left(2+(-1)^{\iota} \tfrac{ix}{2\pi D}\right)^{B}e^{-x}e(xZ)dxds\\
&=\left(\tfrac{i}{2\pi Z}\right)^{A+1}\int_0^{\infty}x^{A}\left(2+(-1)^{\iota} \tfrac{ix}{2\pi Z}\right)^{B}e^{-x}dx\\
&+\tfrac{1}{2\pi i}\int_{(-\sigma_1)}\tilde{\phi}(s)\kappa^s\left\{\left(\tfrac{i}{2\pi Z}\right)^{A-s+1}\int_0^{\infty}x^{A-s}\left(2+(-1)^{\iota} \tfrac{ix}{2\pi Z}\right)^{B}e^{-x}dx\right\}ds.
\end{align*}
The lemma follows.

\end{proof}

Next lemma is purely elementary and is included because it will come up repeatedly in lemma \ref{lemasymp2}.

\begin{lemma}\label{lemextra}For any constants $\alpha,\beta,\gamma\in\mathbb{C}$, and for any $K>1$,
\[\int_{K}^{\infty}x^{\alpha}\left(2+\tfrac{\gamma x}{K}\right)^{\beta}e^{-x}dx=O(e^{-\frac K2}),\]
where the implied constant is independent of $K$. 
\end{lemma}

\begin{proof}First note that the integral converges for every $\alpha,\beta$ and $\gamma$. Changing variables to $x\mapsto x+K$ gives
\begin{align*}
e^{-K}\int_{0}^{\infty}(x+K)^{\alpha}\left(2+\gamma+\tfrac{\gamma x}{K}\right)^{\beta}e^{-x}dx&=\int_0^{K}\cdots dx+\int_{K}^{\infty}\cdots dx\\
&=e^{-K}O\left( K^{\alpha-\beta}\int_0^{K}e^{-x}dx+K^{-\beta}\int_{K}^{\infty}x^{\alpha+\beta}e^{-x}dx\right)\\
&=O(e^{-\frac K2}).
\end{align*}
The implied constant depends only on $\alpha,\beta$ and $\gamma$.
\end{proof}

The following lemma establishes the asymptotic expansion of a certain integral which will come up in the asymptotic expansions of Fourier transforms in theorems \ref{thmasymp} and \ref{thmasymp2}. Its expansion is fundamental for the analysis of the Fourier transforms and the main content of the lemma is the independence of the implied constants of the parameters $C$ and $D$.

\begin{lemma}\label{lemasymp2}Let $Z\in\mathbb{R}\backslash \{0\}$, $\frac12>\kappa>0$, $a,\iota\in\{0,1\}$, and $\mathfrak{R}=\{r_m\}_{m=M_0}^{\infty}$ be a sequence of complex numbers. Then, for any $a\in\mathbb{C}$ with $\Re(a)>-2$, $M_1\geq M_0\in\mathbb{N}$, and $\tau,\tau_1\in\mathbb{R}_{>0}$ we have,
\begin{align*}
\sum_{\substack{m=M_0}}^{M_1}\tfrac{r_m}{2\pi i}\int_{(\tau)}\tfrac{\tilde{\Phi}(u)}{C^u}\int_{0}^1x^{m+\frac{a+u}{2}}(2+(-1)^\iota x)^{\frac{u}{2}}\phi\left(\tfrac{x}{\kappa}\right)e( xZ)dxdu&=\sum_{\substack{m=M_0}}^{M_1}\tfrac{\mathcal{T}_{\mathfrak{R},m,a}^{\iota,\tau}(\Phi)(C^2Z)}{Z^{m+1+\frac{a}{2}}}\\
&\hspace{0.6in}+O((C^2Z)^{-\tau_1}Z^{-(M_1+2+\frac{a}{2})}),
\end{align*}
where
\[\mathcal{T}_{\mathfrak{R},m,a}^{\iota,\tau}(\Phi)(x):=\tfrac{1}{2\pi i}\int_{(\tau)}\tilde{\Phi}(u)r_{\iota,m,a}\left(\tfrac{u}{2}\right)\Gamma\left(m+1+\tfrac{a+u}{2}\right)x^{-\frac{u}{2}}du,\]
and
\[r_{\iota,m,a}(y):=\left(\tfrac{i}{2\pi}\right)^{m+1+y+\frac{a}{2}}2^y\sum_{\substack{j+k=m\\ k\geq M_0\\ j\geq0}}\tfrac{r_k}{((-1)^{\iota}2)^j}\binom{y}{j}.\]
Moreover the implied constant depends only on $\mathfrak{R},\Phi,M_1$, and $\tau_1$, and in case $\tilde{\Phi}(u)$ has at most a simple pole at $u=0$, one can take $\tau_1=0$.
\end{lemma}

\begin{proof}Note that the equality is trivially true if $|Z|<1$, since then the error term dominates, therefore we can assume that $|Z|>1$ for the proof. 

Using lemma \ref{lemasymp} in the $x$-integral we can rewrite our integral as
\begin{align*}
&\sum_{\substack{m=M_0}}^{M_1}\tfrac{r_m}{2\pi i } \int_{(\tau)}\tfrac{\tilde{\Phi}(u)}{C^u}\left\{\left(\tfrac{i}{2\pi Z}\right)^{m+1+\frac{a+u}{2}}\int_0^{\infty}x^{m+\frac{a+u}{2}}\left(2+(-1)^{\iota}\tfrac{ix}{2\pi Z}\right)^{\frac{u}{2}}e^{-x}dx\right.\tag{$i$}\label{aplemi}\\
&+\tfrac{1}{2\pi i}\int_{(-\sigma_1)}\tilde{\phi}(s)\kappa^s\left[\left(\tfrac{i}{2\pi Z }\right)^{m-s+1+\frac{a+u}{2}}\int_0^{\infty}x^{m-s+\frac{a+u}{2}}\left(2+(-1)^{\iota} \tfrac{ix}{2\pi Z}\right)^{\frac{u}{2}}e^{-x}dx\right]ds\tag{$ii$}\label{aplemii}\\
&-\left.\tfrac{1}{2\pi i}\int_{(\sigma_2)}\tilde{\phi}(s)\kappa^s\int_1^{1+\frac{|Z|}{ Z}i\infty}x^{m-s+\frac{a+u}{2}}(2+(-1)^{\iota} x)^{\frac{u}{2}}e( xZ)dxds\right\}du\tag{$iii$}\label{aplemiii}
\end{align*}
Since $m\geq0$ and $\Re(a)>-2$ each integral in \eqref{aplemi}, \eqref{aplemii}, and  \eqref{aplemiii} is convergent and homorphic in the $u$-variable for $\Re(u)\geq 0$. Since the function $\tilde{\Phi}(u)$ is also holomorphic in $\Re(u)>0$ and has a simple pole at $u=0$ (cf. lemma 3.3 of \cite{Altug:2015aa}) we are free to move the $u$-contour in $\Re(u)>0$ (in case $\tilde{\Phi}(u)$ has only a simple pole at $u=0$ we can even move the contour to $\Re(u)=0$ by taking principal value of the integral). Therefore we can shift the $u$-contour to $\Re(u)=\tau$ for any $\tau>0$ for the main terms, and to $\Re(u)=\tau_1$ for the error terms.

We will now analyze each of the terms \eqref{aplemi}, \eqref{aplemii}, and \eqref{aplemiii} separately. The main term will come from $\eqref{aplemi}$ and the rest will contribute to the error.

\begin{itemize}
\item \eqref{aplemi}. This is the most complicated part of the analysis and it gives the main contribution. The analysis will have two separate parts. We will first analyze the $x$-integral, then substitute the result in the $u$-integral and move the $u$-contour to get the result. We remark that in order to shift the $u$-contour we need to make sure that the error terms are holomorphic in the variable $u$. Below we will give explicit formulas for the error terms which will show holomorphy, and then we will give bounds on each, depending on $u$. These bounds will then be used to estimate the error terms.

Let us start by breaking the region of integration into two so that we can use the binomial theorem on the $(2+(-1)^{\iota} ix/2\pi Z)$-factor.
\begin{align*}
\int_0^{\infty}x^{m+\frac{a+u}{2}}\left(2+(-1)^{\iota} \tfrac{ix}{2\pi Z}\right)^{\frac{u}{2}}e^{-x}dx&=\int_0^{|Z|}\cdots dx+\int_{|Z|}^{\infty}\cdots dx.
\end{align*}
By lemma \ref{lemextra} the second integral is exponentially small in $|Z|$, i.e.
\begin{align*}
\int_{|Z|}^{\infty}x^{m+\frac{a+u}{2}}\left(2+(-1)^{\iota} \tfrac{ix}{2\pi Z}\right)^{\frac{u}{2}}e^{-x}dx&=O(e^{-\frac{|Z|}{2}}),\tag{$i$-1}\label{aplemi1}
\end{align*}
where the implied constant depends only on $m,a$, and $u$. This ends the analysis of the second integral. For the first integral we use the binomial theorem,
\begin{multline*}
\int_0^{|Z|}x^{m+\frac{a+u}{2}}\left(2+(-1)^{\iota} \tfrac{ix}{2\pi Z}\right)^{\frac{u}{2}}e^{-x}dx=\sum_{j=0}^{M_1+1}\tfrac{\nu_j^{(\iota)}(u)}{Z^j}\int_0^{|Z|}x^{m+j+\frac{a+u}{2}}e^{-x}dx\\
+\int_0^{|Z|}\tfrac{x^{m+M_1+2+\frac{a+u}{2}}R_{M_1}^{(\iota)}(u,Z,x)}{Z^{M_1+2}}e^{-x}dx,\tag{$i$-2}\label{aplemi2}
\end{multline*}
where
$$R_{M_1}^{(\iota)}(s, Z,x):=\left(\tfrac{Z}{x}\right)^{M_1+2}\left\{\left(2+(-1)^{\iota}\tfrac{ix}{2\pi  Z}\right)^{\frac{s}{2}}-\sum_{j=0}^{M_1+1}\nu_j^{(\iota)}(s)\left(\tfrac{x}{Z}\right)^{j}\right\},$$
 and $\nu_j^{(\iota)}(s):=2^{\frac{s}{2}}\binom{s/2}{j}\left(\frac{(-1)^\iota i}{4\pi }\right)^{j}$. Note that $R_{M_1}^{(\iota)}(s,Z,x)$ is a holomorphic function of $s$, and by the binomial theorem, for $0<x<|Z|$, $|R_{M_1}^{(\iota)}(s,Z,x)|<3^{\Re(s)/2}$. Furthermore, we have,
\begin{align*}
\int_0^{|Z|}x^{m+j+\frac{a+u}{2}}e^{-x}dx&=\Gamma\left(m+j+1+\tfrac{a+u}{2}\right)+\int_{|Z|}^{\infty}x^{m+j+\frac{a+u}{2}}e^{-x}dx\\
&\hspace{-1in}=\Gamma\left(m+j+1+\tfrac{a+u}{2}\right)+O(e^{-\frac{|Z|}{2}}).\tag{$i$-3}\label{aplemi3}
\end{align*}
The implied constant above depends only on $m, u$, and $a $. We now take \eqref{aplemi1}, \eqref{aplemi2}, and \eqref{aplemi3} and substitute them in \eqref{aplemi}.

Note that the sum of \eqref{aplemi1} and the error terms in \eqref{aplemi2} and \eqref{aplemi3} is $O(|Z|^{-(M_1+2+\frac{\delta}{2})})$. Using this bound (and for each $m\geq0$ grouping the terms whose indices satisfy $n+j=m$ together) we finally get 
\begin{align*}
\eqref{aplemi}&=\sum_{\substack{m=M_0}}^{M_1}\tfrac{\mathcal{T}_{\mathfrak{R},m,a}^{\iota,\tau}(\Phi)(C^2Z)}{Z^{m+1+\frac{a}{2}}}+O((C^2Z)^{-\tau_1}Z^{-(M_1+2+\frac{a}{2})}),\tag{$i$-4}\label{aplemi4}
\end{align*}
where $\mathcal{T}_{\mathfrak{R},m,a}^{\iota,\tau}(C^2Z)$ is as defined in the statement of the proposition, $\tau>0,\tau_1\geq0$, and the implied constant depends only on $\tau_1,M_1,\phi,\mathfrak{R},a$, and $\Phi$.

\item \eqref{aplemii}. We bound this term by moving the $s$-contour to $\Re(s)=-M_1$. As before, we break the integral into two parts,
\begin{align*}
\int_0^{\infty}x^{m-s+\frac{a+u}{2}}\left(2+(-1)^{\iota} \tfrac{ix}{2\pi Z}\right)^{\frac{u}{2}}e^{-x}dx=\int_0^{|Z|}\cdots dx+\int_{|Z|}^{\infty}\cdots dx,
\end{align*}
and the second integral is $O(e^{-\frac{|Z|}{2}})$ by lemma \ref{lemextra}. Once again the constant depends only on $m,a$, and $u$. Moving the $s$-contour to $-(M_1+1)$ (recall that we are assuming $|Z|>1$) shows that the first integral is bounded by $\Gamma\left(M_1+2+\frac{u+a}{2}\right)$. Since we can move the $u$-contour to any $\Re(u)=\tau_1\geq0$ this shows that $\eqref{aplemi2}=O((C^2Z)^{-\tau_1}Z^{-(M_1+2+\frac a2)})$, where the implied constant depends only on $\tau_1,M_1,a,\mathfrak{R}$, and $\Phi$.

\item \eqref{aplemiii}. To bound this term we  use integration by parts $\mu(M_1,a+u):=\lfloor M_1+2+\frac{\Re(a+u)}{2}\rfloor$-times on the $x$-integral. This gives,
\begin{align*}
&\tfrac{1}{2\pi i}\int_{(\sigma_2)}\tilde{\phi}(s)\kappa^s\int_1^{1+\frac{|Z|}{ Z}i\infty}x^{m-s+\frac{a+u}{2}}(2+(-1)^{\iota} x)^{\frac{u}{2}}e( xZ)dxds\\
&\hspace{0.4in}=\tfrac{1}{2\pi i( 2\pi i Z)^{\mu(M_1,u+a)}}\int_{(\sigma_2)}\tilde{\phi}(s)\kappa^s\int_1^{1+\frac{|Z|}{ Z}i\infty}\tfrac{d^{\mu(M_1,u+a)}}{dx^{\mu(M_1,u+a)}}\left\{x^{m-s+\frac{u+a}{2}}(2+(-1)^{\iota} x)^{\frac{u}{2}}\right\}e( xZ)dxds.
\end{align*}
Below, we will justify that the boundary terms vanish. Taking this for granted for the moment we bound the integrands trivially. Note that $\mu(M_1,u+a)=\lfloor M_1+2+\Re\left(\frac{u+a}{2}\right)\rfloor=\lfloor M_1+2+\Re\left(\frac{\tau_1+a}{2}\right)\rfloor$. This implies that the contribution of \eqref{aplemiii} is $O(Z^{-\lfloor M_1+2+\Re\left(\frac{\tau_1+a}{2}\right)\rfloor}C^{-\tau_1})=O(Z^{-\left( M_1+1+\Re\left(\frac{a}{2}\right)\right)}(C^2Z)^{-\frac{\tau_1}{2}})$, where the implied constant depends only on $\Phi,\phi,M_1, \tau_1,a$ and $\kappa$.

The only point left to justify is the  vanishing of the boundary terms in the integration by parts. The boundary term at $1+\frac{|Z|}{ Z}i\infty$ vanish because of the exponential factor. The other boundary term vanishes because $\phi^{(k)}(1/\kappa)=0$ for every $k\in\mathbb{N}$. More precisely, let $(s-\beta)_k=(s-\beta)(s-\beta-1)\cdots(s-\beta-k+1)$ be the falling factorial. Then, for any $\beta\in\mathbb{C}$ and $k\geq1$,
\[(s-\beta)_k\tilde{\phi}(s)=\int_0^\infty(\phi(y)y^{k+\beta-1})^{(k)}y^{s-1}dy\]
Therefore,
\[\tfrac{1}{2\pi i}\int_{(\sigma_2)}\tilde{\phi}(s)\kappa^s(\beta-s)_kx^{\beta-s-k}ds=(-1)^kx^{\beta-k}\left.\tfrac{d^{k}}{dy^k}\left\{\phi(y)y^{k+\beta-1}\right\}\right|_{y=\frac x\kappa}.\]
Since $\phi$ and all its derivatives vanish at $x=1/\kappa$ (by definition \ref{apdef2}) all the boundary terms vanish. 
\end{itemize}

\end{proof}

\paragraph{Asymptotic expansions.}

In theorems \ref{thmasymp} and \ref{thmasymp2} we develop the asymptotic expansion of Fourier transforms of functions with certain prescribed singularities. Once again, independence of the implied constants of the parameters $C$ and $D$ is the central issue.

\begin{thm}\label{thmasymp} Let $C,D\in\mathbb{R}\backslash\{0\}$ and $\Phi\in \mathcal{S}(\mathbb{R})$, $a\in\mathbb{C}$ with $\Re(a)>-2$, and $h_a(x)$ be
\[h_a(x)=|1-x^2|^{\frac a2}h_1(x),\]
where $h_1(x)$ is smooth in and up-to the boundary of $(-1,1)$. Assume that around $x=\pm1$ it has an asymptotic expansion
\[h_a(\pm(1-x))\sim |x|^{\frac a2}\sum_{m=0}^{\infty}c_m^{\pm}x^m.\]
Then, for any $\tau,\tau_1\in\mathbb{R}_{>0}$ and  $M\in\mathbb{Z}_{>0}$ we have,
 \begin{align*}
 \int_{-1}^1h_a(x)\Phi\left(\tfrac{C}{\sqrt{1-x^2}}\right)e(xD)dx&=\sum_{\substack{m=0\\ \pm}}^{M}\tfrac{e(\pm D)\mathcal{A}_{h_a,m}^{\tau,\pm}(\Phi)(\mp C^2D)}{(\mp D)^{m+1+\frac{a}{2}}}+O((C^2D)^{-\tau_1}D^{-(M+2+\frac{a}{2})}),
 \end{align*}
where $\mathcal{A}_{h_a,m}^{\tau,\pm}(\Phi)(x)$ is as in definition \ref{inttrans}. Moreover, the implied constant depends only on $h,\Phi,M,\tau_1,\epsilon$, and $a$, and in case $\tilde{\Phi}(u)$ has at most a simple pole at $u=0$, one can take $\tau_1=0$. 
\end{thm}

\begin{proof}The proof is technical but straightforward. Integrand has singularities around $\pm1$ which give the main term and the rest is absorbed in the error.

Let $\phi_{\kappa}(x)$ be a cut-off function as in definition \ref{sing27}. Then,
\[\int_{-1}^1h_a(x)\cdots dx=\int_{-1}^1h_a(x)(1-\phi_{\kappa}(x))\cdots dx+\int_{-1}^{1}h_a(x)\phi_{\kappa}(x)\cdots dx.\]
Since $h_a(x)$ is smooth in $(-1,1)$, by lemma \ref{aplem4} we see that for any $M_0,N_0\geq0$ the first integral is $O(C^{-N_0}D^{-M_0})$ (and the implied constant is independent of $C,D$). Choosing $M_0=M+\tau_1+2+\frac{a}{2}, N_0=2\tau_1$ we get that the first integral is $O((C^2D)^{-\tau_1}D^{-(M+2+\frac{a}{2})})$.

We go on with the analysis of the second term which will give us the main contribution. The main idea is to use the asymptotic expansion and lemma \ref{lemasymp2}. In order to save space we will go through the argument at once for both points $x=\pm1$. Let $\kappa\,(=\kappa_M)$ be such that $\frac12>\kappa>0$ and
\begin{equation*}
h_a(\pm(1-x))=|x|^{\frac a 2}\sum_{m=0}^{M+2}c_m^{\pm}x^{m}+|x|^{\frac{a}{2}}x^{M+3}h_{a,M+3}^{\pm}(x)
\end{equation*}
for all $| x|<\kappa$, where $h_{a,M+3}^{\pm}(x)$ is smooth in and up to the boundary of $[0,\kappa]$.

Substituting this expansion into the integral,
\begin{align*}
\int_{-1}^1h_a(x)(\cdots)dx&=\int_{0}^{1}h_a(x)(\cdots)dx+\int_{-1}^0h_a(x)(\cdots)dx\\
&=\sum_{\pm}e(\pm D)\int_{0}^{1}h_a(\pm(1-x))\phi\left(\tfrac{x}{\kappa}\right)\Phi\left(\tfrac{C}{\sqrt{x(2-x)}}\right)e(\mp xD)dx\\
&=\sum_{\substack{m=0\\ \pm}}^{M+2}e(\pm D)c_m^{\pm}\int_{0}^{1}x^{m+\frac{a}{2}}\phi\left(\tfrac{x}{\kappa}\right)\Phi\left(\tfrac{C}{\sqrt{x(2-x)}}\right)e(\mp xD)dx\\
&\hspace{0.4in}+\sum_{\substack{\pm}}e(\pm D)\int_{0}^{1}h_{a,M+3}^{\pm}(x)x^{M+3+\frac{a}{2}}\phi\left(\tfrac{x}{\kappa}\right)\Phi\left(\tfrac{C}{\sqrt{x(2-x)}}\right)e(\mp xD)dx.
\end{align*}
For any $\tau\in\mathbb{R}_{>0}$, using Mellin inversion on $\Phi$ and the definition of $\phi_{\kappa}$ given in \eqref{sing27} we see that the above integrals can be written as
\begin{align*}
&\sum_{\substack{m=0\\ \pm}}^{M+2}e(\pm D)c_m^{\pm}\int_{0}^1x^{m+\frac{a}{2}}\phi\left(\tfrac{x}{\kappa}\right)\left\{\tfrac{1}{2\pi i}\int_{(\tau)}\tfrac{\tilde{\Phi}(u)(x(2-x))^{\frac{u}{2}}}{C^u}du\right\}e(\mp xD)dx\tag{$*$}\label{aplem*}\\
&\hspace{0.4in}+\sum_{\pm}e(\pm D)\int_{0}^1h_{a,M+3}^{\pm}(x)x^{M+3+\frac{a}{2}}\phi\left(\tfrac{x}{\kappa}\right)\left\{\tfrac{1}{2\pi i}\int_{(\tau)}\tfrac{\tilde{\Phi}(u)(x(2-x))^{\frac{u}{2}}}{C^u}du\right\}e(\mp xD)dx.\tag{$**$}\label{aplem**}
\end{align*}
We then interchange the $u$ and $x$-integrals, which is justified since $\tilde{\Phi}(u)$ decays faster than any polynomial, $m,\tau>0$, and $\Re(a)>-2$ so that we have $m+(a+\tau)/2>-1$, therefore the double integral converges absolutely. This gives,
\begin{multline*}
\sum_{\substack{m=0\\\pm}}^{M+2}\tfrac{e(\pm D)c_m^{\pm}}{2\pi i}\int_{(\tau)}\tfrac{\tilde{\Phi}(u)}{C^u}\int_{0}^1x^{m+\frac{a+u}{2}}(2-x)^{\frac{u}{2}}\phi\left(\tfrac{x}{\kappa}\right)e(\mp xD)dxdu\\
+\sum_{\pm}\tfrac{e(\pm D)}{2\pi i}\int_{(\tau)}\tfrac{\tilde{\Phi}(u)}{C^u}\int_{0}^1h_{a,M+3}^{\pm}( x)x^{M+3+\frac{a+u}{2}}(2-x)^{\frac{u}{2}}\phi\left(\tfrac{x}{\kappa}\right)e(\mp xD)dxdu.
\end{multline*}
For the first line, lemma \ref{lemasymp2} (with $r_m=c_m^{\pm}$, $Z=\mp D$, $M_0=0$, $M_1=M+2$, and $\iota=1$) gives, 
\begin{equation}\label{propaympeq1}
\eqref{aplem*}=\sum_{\substack{m=0\\ \pm}}^{M}\tfrac{e(\pm D)\mathcal{A}_{h_a,m}^{\tau,\pm}(\Phi)(\mp C^2D)}{(\mp D)^{m+1+\frac{a}{2}}}+O((C^2D)^{-\tau_1}D^{-(M+2+\frac{a}{2})}),
\end{equation}
where the implied constant depends only on $h_a,\Phi,M,\tau_1,\epsilon$ and $\delta$. For the second line let $\mu_a(M,u+\delta):=\lfloor M+3+\Re\left(\frac{a+u+\delta}{2}\right)\rfloor$. To bound \eqref{aplem**}, we follow the proof of lemma \ref{lemasymp2} and use integration by parts $\mu_a(M,u+\delta)$-times (For this, recall that $h_{a,M+3}^{\pm}(x)$ is smooth inside and up to the boundary of $(-1,1)$ and $\phi(x/\kappa)$ and all its derivatives vanish at $x=1$.). This gives the following expression for \eqref{aplem**},
\begin{equation}\label{propaympeq2}
\sum_{\pm}\tfrac{e(\pm D)}{2\pi i}\int_{(\tau)}\tfrac{\tilde{\Phi}(u)}{C^u(\mp D)^{\mu_a(M,u+\delta)}}\int_0^{1}\tfrac{d^{\mu_a(M,u+\delta)}}{dx^{\mu_a(M,u+\delta)}}\left\{x^{M+3+\frac{a+u+\delta}{2}}(2-x)^{\frac{u+\delta}{2}}h_{a,M+3}^{\pm}(\mp x)\phi\left(\tfrac{x}{\kappa}\right)\right\}e(\mp xD)dxdu.
\end{equation}
(Note that the boundary terms vanish because $\phi(1/\kappa)=0$ and $\Re(M+3+\frac{1+u+\delta}{2}-\mu(M,u+\delta))>0$.) 

Bounding the $u$ and $x$-integrals trivially and combining \eqref{propaympeq1} with \eqref{propaympeq2} finishes the proof.

\end{proof}

\begin{thm}\label{thmasymp2} Let $C,D\in\mathbb{R}\backslash\{0\}$ and $\Phi\in \mathcal{S}(\mathbb{R})$. Let $a\in\mathbb{C}$ with $\Re(a)>-2$ we have, and $h_a(x)$ be
\[h_a(x)=|1-x^2|^{\frac a2}h_1(x),\]
where $h_1(x)$ is compactly supported and smooth in and up-to the boundary of $\begin{setdef}{x\in Supp(h_a)}{|x|>1}\end{setdef}$. Assume that around $x=\pm1$ it has an asymptotic expansion 
\[h_a(\pm(1-x))\sim |x|^{\frac a2}\sum_{m=0}^{\infty}c_m^{\pm}x^m.\]
Then, for any $\tau,\tau\in\mathbb{R}_{>0}$ and $M\in\mathbb{Z}_{>0}$ we have,
 \begin{align*}
 \int_{|x|>1}h_a(x)\Phi\left(\tfrac{C}{\sqrt{x^2-1}}\right)e(xD)dx&=\sum_{\substack{m=0\\ \pm}}^{M}\tfrac{e(\pm D)(-1)^m\mathcal{A}_{h_a,m}^{\tau,\pm}(\Phi)(\pm C^2D)}{(\pm D)^{m+1+\frac{a}{2}}}+O((C^2D)^{-\tau_1}D^{-(M+2+\frac{a}{2})}),
 \end{align*}
 where $\mathcal{A}_{h_a,m}^{\tau,\pm}(\Phi)(x)$ is in definition \ref{inttrans} and the implied constant depends only on $h,\Phi,M,\tau_1,\epsilon$ and $a$, and in case $\tilde{\Phi}(u)$ has at most a simple pole at $u=0$, one can take $\tau_1=0$. 
\end{thm}

\begin{proof} The proof is identical to the proof of theorem \ref{thmasymp}, we just need to keep track of various signs.

Let $\phi_{\kappa}(x)$ be a cut-off function as in definition \ref{sing27}. Then,
\[\int\cdots dx=\int(1-\phi_{\kappa}(x))\cdots dx+\int\phi_{\kappa}(x)\cdots dx.\]
Lemma \ref{aplem4} implies that for any $M_0,N_0\geq0$ the first integral is $O(C^{-N_0}D^{-M_0})$ (and the implied constant is independent of $C$ and $D$). Choosing $M_0=M+\tau_1+2+\frac{a}{2}, N_0=2\tau_1$ we get that the first integral is $O((C^2D)^{-\tau_1}D^{-(M+2+\frac{a}{2})})$.

Let $\kappa\,(=\kappa_M)$ be such that $\frac12>\kappa>0$ and
\begin{equation*}
h_a(\pm(1-x))=|x|^{\frac a 2}\sum_{m=0}^{M+2}c_m^{\pm}x^{m}+|x|^{\frac a2}x^{M+3}h_{a,M+3}^{\pm}(x)
\end{equation*}
for all $| x|<\kappa$, where $h_{a,M+3}^{\pm}(x)$ is smooth. Substituting this expansion into the integral,
\begin{align*}
\int h_a(x)\phi_{\kappa}(x)(\cdots)dx&=\int_{1}^{2}\cdots dx+\int_{-2}^{-1}\cdots dx\\
&=\sum_{\pm}e(\pm D)\int_{-1}^{0}h_a(\pm(1-x))\phi\left(\tfrac{x}{\kappa}\right)\Phi\left(\tfrac{C}{\sqrt{x^2-2x}}\right)e(\mp xD)dx\\
&=\sum_{\substack{m=0\\ \pm}}^{M+2}e(\pm D)(-1)^{m}c_m^{\pm}\int_{0}^{1}x^{m+\frac{a}{2}}\phi\left(\tfrac{x}{\kappa}\right)\Phi\left(\tfrac{C}{\sqrt{x(2+x)}}\right)e(\pm xD)dx\\
&+\sum_{\substack{\pm}}e(\pm D)(-1)^{\frac a2}\int_{-1}^{0}h_{a,M+3}^{\pm}(x)x^{M+3+\frac a2}\phi\left(\tfrac{x}{\kappa}\right)\Phi\left(\tfrac{C}{\sqrt{x^2-2x}}\right)e(\mp xD)dx.
\end{align*}
For any $\tau\in\mathbb{R}_{>0}$, using Mellin inversion on $\Phi$ and the definition of $\phi_{\kappa}$ given in \eqref{sing27} we see that the above integrals can be written as
\begin{align*}
&\sum_{\substack{m=0\\ \pm}}^{M+2}e(\pm D)(-1)^mc_m^{\pm}\int_{0}^1x^{m+\frac{a}{2}}\phi\left(\tfrac{x}{\kappa}\right)\left\{\tfrac{1}{2\pi i}\int_{(\tau)}\tfrac{\tilde{\Phi}(u)(x(2+x))^{\frac{u}{2}}}{C^u}du\right\}e(\pm xD)dx\tag{$*$}\label{aplem2**}\\
&+\sum_{\pm}e(\pm D)(-1)^{\frac a2}\int_{-1}^0h_{a,M+3}^{\pm}(x)x^{M+3+\frac a2}\phi\left(\tfrac{x}{\kappa}\right)\left\{\tfrac{1}{2\pi i}\int_{(\tau)}\tfrac{\tilde{\Phi}(u)(x^2-2x)^{\frac{u}{2}}}{C^u}du\right\}e(\mp xD)dx.\tag{$**$}\label{aplem2***}
\end{align*}

Next, we interchange the $u$ and $x$-integrals. The interchange is justified since $\tilde{\Phi}(u)$ decays faster than any polynomial, $m,\tau>0$, and $a>-2$ so that we have $m+(\tau+a)/2>-1$, therefore the double integral converges absolutely. This gives,
\begin{align*}
&\sum_{\substack{m=0\\\pm}}^{M+2}\tfrac{e(\pm D)(-1)^mc_m^{\pm}}{2\pi i}\int_{(\tau)}\tfrac{\tilde{\Phi}(u)}{C^u}\int_{0}^1x^{m+\frac{u+a}{2}}(2+x)^{\frac{u}{2}}\phi\left(\tfrac{x}{\kappa}\right)e(\pm xD)dxdu\\
&+\sum_{\pm}\tfrac{e(\pm D)(-1)^{\frac a 2}}{2\pi i}\int_{(\tau)}\tfrac{\tilde{\Phi}(u)}{C^u}\int_{-1}^0h_{a,M+3}^{\pm}( x)x^{M+3+\frac a 2}(x^2-2x)^{\frac{u}{2}}\phi\left(\tfrac{x}{\kappa}\right)e(\mp xD)dxdu.
\end{align*}
As in the proof of theorem \ref{thmasymp}, for the first line lemma \ref{lemasymp2} (with $r_m=(-1)^mc_m^{\pm}$, $Z=\pm D$, $M_0=0$, $M_1=M+2$, and $\iota=0$) gives, 
\begin{equation}\label{prop2aympeq2}
\eqref{aplem2**}=\sum_{\substack{m=0\\ \pm}}^{M}\tfrac{e(\pm D)(-1)^m\mathcal{A}_{h_a,m}^{\tau,\pm}(\Phi)(\pm C^2D)}{(\pm D)^{m+1+\frac{a}{2}}}+O((C^2D)^{-\tau_1}D^{-(M+2+\frac{a}{2})}),
\end{equation}
where the implied constant depends only on $h,\Phi,M,\tau_1,\epsilon$ and $a$. For \eqref{aplem**} line let $\mu(M,u+a):=\lfloor M+3+\Re\left(\frac{u+a}{2}\right)\rfloor$ and use integration by parts $\mu(M,u+a)$-times (Recall that $h_{a,M+3}^{\pm}(x)$ is smooth and $\phi(x/\kappa)$ and all its derivatives vanish at $x=1$.). This gives the following expression for \eqref{aplem2***},
\begin{equation}\label{prop2aympeq3}
\sum_{\pm}\tfrac{e(\pm D)(-a)^{\frac a2}}{2\pi i}\int_{(\tau)}\tfrac{\tilde{\Phi}(u)}{C^u(\mp D)^{\mu(M,u+a)}}\int_{-1}^0\tfrac{d^{\mu(M,u+a)}}{dx^{\mu(M,u+a)}}\left\{x^{M+3+\frac{1+u+a}{2}}(2-x)^{\frac{u}{2}}h_{a,M+3}^{\pm}( x)\phi\left(\tfrac{x}{\kappa}\right)\right\}e(\mp xD)dxdu.
\end{equation}
(Note that the boundary terms vanish because $\phi(1/\kappa)=0$ and $\Re(M+3+\frac{u+a}{2}-\mu(M,u+a))>0$.) Bounding the $u$ and $x$-integrals trivially, and combining \eqref{prop2aympeq2} with \eqref{prop2aympeq3} finishes the proof.

\end{proof}

\begin{cor}\label{corasymp15} Let $C,D\in\mathbb{R}\backslash\{0\}$, $\Phi\in\mathcal{S}(\mathbb{R})$ such that $\tilde{\Phi}(u)$ is holomorphic for $\Re(u)>0$, and $h(x)\in C_c(\mathbb{R})$ be such that $h(x)=|1-x^2|^{\frac{a}{2}}h_1(x)$, where $h_1(x)\in C_c^{\infty}(\mathbb{R})$ and $a\in\mathbb{Z}_{\geq -2}$. Assume further that $h(x)$ has the following asymptotic expansion around $x=\pm1$,
\[h(\pm(1-x))\sim |x|^{\frac a2}\sum_{m=0}^{\infty}c_m^{\pm}x^{m}.\]
Then for any $N>0$,
\begin{equation}\label{apcor1515}
\int_{\mathbb{R}}h(x)\Phi\left(\tfrac{C}{\sqrt{|1-x^2|}}\right)e(xD)dx=O((C^2D)^{-N}D^{-(1+\frac a2)}),
\end{equation}
where the implied constant depend only on $h(x)$ and $\Phi$. If $\tilde{\Phi}(u)$ is also holomorphic for $\Re(u)\geq-2$ with at most simple poles at $u=0$ and $u=-2$, we furthermore have
\begin{equation}\label{apcor1516}
\int_{\mathbb{R}}h(x)\Phi\left(\tfrac{C}{\sqrt{|1-x^2|}}\right)e(xD)dx=\sum_{\pm}\tfrac{e(\pm D)c_0^{\pm,a}(\Phi)(1+(-1)^{1-\frac{a}{2}})}{(\pm D)^{1+\frac a2}}+O(C^2D^{-\frac{a}{2}}+D^{-(2+\frac a2)}),
\end{equation}
where
$$c_0^{\pm,a}(\Phi)=c_0^{\pm}(0)\Gamma\left(1+\frac{a}{2}\right)\text{Res}_{u=0}\tilde{\Phi}(u),$$ 
and the implied constant in the error term depend only on $h(x),\Phi$ and $a$. Note also that if $a\equiv 0\bmod 4$ then the leading term vanishes.
\end{cor}

\begin{proof} Divide the integral according to $|x|<1$ and $|x|>1$. By theorems \ref{thmasymp} and \ref{thmasymp2} (taking $M=0$ in both), for any $\tau_1\geq0$ we have
\begin{align*}
\int_{-1}^1h(x)\Phi\left(\tfrac{C}{\sqrt{1-x^2}}\right)e(xD)dx&=\sum_{\substack{ \pm}}\tfrac{e(\pm D)\mathcal{A}_{h_a,0}^{\tau,\pm}(\Phi)(\mp C^2D)}{(\mp D)^{1+\frac{a}{2}}}+O((C^2D)^{-\tau_1}D^{-(2+\frac{a}{2})})\tag{$*$}\label{apcor15*},\\
\int_{|x|<1}h(x)\Phi\left(\tfrac{C}{\sqrt{x^2-1}}\right)e(xD)dx&=\sum_{\substack{\pm}}\tfrac{e(\pm D)\mathcal{A}_{h_a,0}^{\tau,\pm}(\Phi)(\pm C^2D)}{(\pm D)^{1+\frac{a}{2}}}+O((C^2D)^{-\tau_1}D^{-(2+\frac{a}{2})})\tag{$**$}\label{apcor15**}.
\end{align*}
By proposition \ref{lemasymp3} $\mathcal{A}_{h_a,0}^{\tau,\pm}(C^2D)=(C^2D)^{-\tau}$, and since $\tilde{F}(u)$ is holomorphic for $\Re(u)>0$ we can move the $u$-contour in the formula for $\mathcal{A}_{h_a,0}^{\tau,\pm}$ (cf. definition \ref{inttrans}) to $\Re(u)=\tau$ for any $\tau\in \mathbb{R}_{>0}$. In particular choosing $\tau=\tau_1=N$ gives \eqref{apcor1515}. 

For the second estimate assume that $\tilde{\Phi}(u)$ is holomorphic in $\Re(u)\geq-2$ with only possible poles at $u=0,2$, both simple. Then we shift the $u$-contour in the defintion of $\mathcal{A}_{h_a,0}^{\tau,\pm}(\Phi)(\mp C^2D)$ to $\Re(u)=-2$ (We take the principal part of the integral in case there is a pole at $u=-2$, but this does not effect the bound.). This picks up the residue at $u=0$ and gives,
\begin{align*}
\mathcal{A}_{h_a,0}^{\tau,\pm}(\Phi)(x)&=\tfrac{1}{2\pi i }\int_{(\tau)}\tilde{\Phi}(u)c_{0}^{\pm}\left(\tfrac{u}{2}\right)\Gamma\left(1+\tfrac{a+u}{2}\right)x^{-\frac{u}{2}}du\\
&=c_0^{\pm}(0)\Gamma\left(1+\tfrac{a}{2}\right)\text{Res}_{u=0}\tilde{\Phi}(u)+O(x)\tag{$\circ$}\label{apcor15o},
\end{align*}
where the implied constant depends only on $h(x)$ and $\Phi$. Substituting \eqref{apcor15o} into \eqref{apcor15*} and \eqref{apcor15**} gives
\[\int_{\mathbb{R}}h(x)\Phi\left(\tfrac{C}{\sqrt{|1-x^2|}}\right)e(xD)dx=\sum_{\pm}\tfrac{e(\pm D)c_0^{\pm,a}(\Phi)(1+(-1)^{1-\frac{a}{2}})}{(\pm D)^{1+\frac a2}}+O(C^2D^{-\frac{a}{2}}+O((C^2D)^{-\tau_1})D^{-(2+\frac a2)}).\]
Finally, choosing $\tau_1=0$ finishes the proof of \eqref{apcor1516}.
\end{proof}

\section{Analysis of character sums}\label{sectioncharsum}

\subsection{Notation.} Before starting the computations let us introduce some notation that will be used throughout the calculations. Let $q$ be a prime. For any integer $A\in\mathbb{Z}$ let $v_q(A)$ denote the $q$-adic valuation of $A$. In what follows we will denote the ``$q$-part'' and the ``prime to $q$-part'' of $A$ respectively by $A_{(q)}$ and $A^{(q)}$. They are defined by,
\[A_{(q)}:=q^{v_q(a)}\hspace{0.5in}and\hspace{0.5in}A^{(q)}:=\tfrac{A}{A_{(q)}}.\]
Let $a,b\in\mathbb{Z}$ with $b>0$. We define $\delta(a;b)$ by,
\[\delta(a;b):=\begin{cases}1& x^2\equiv a \bmod b\text{ has a solution}\\ 0 &\text{otherwise}\end{cases}.\]
We also recall the definition of Kloosterman sums as they will show up in the calculations. Let $a,b\in\mathbb{Z}$, then $S(a,b;q)$ is defined by,
\[S(a,b;q)=\sum_{x\in \mathbb{F}_q^{\times}}e\left(\tfrac{a x+b x^{-1}}{q}\right).\]
Finally, for an integer $\alpha\in\mathbb{Z}\backslash \{0\}$, let $rad(\alpha)$ denote the radical (or the square-free part of) of $\alpha$. i.e.
\[rad(\alpha)=\prod_{\substack{q\mid \alpha\\q-prime}}q.\]

\subsection{Analysis of $Kl_{l,f}(\xi,n)$}

$Kl_{l,f}(\xi,n)$ are close relatives of classical Kloosterman sums. Being such, they satisfy the same twisted multiplicative property that Kloosterman sums satisfy (c.f. equation (1.59) of \cite{Iwaniec:2004aa}).

\begin{lemma}\label{estlem1}Let $l,f\in\mathbb{Z}_{>0}$ and $\xi,n\in\mathbb{Z}$. Then,
\[Kl_{l,f}(\xi,n)=\prod_{q_\mid 4lf^2}Kl_{l_{(q)},f_{(q)}}(((4lf^2)^{(q)})^{-1}\xi,n),\]
where, by abuse of notation, $((4lf^2)^{(q)})^{-1}$ denotes the inverse of $(4lf^2)^{(q)}$ modulo $(4lf^2)_{(q)}$.
\end{lemma}

\begin{proof}
This is a straightforward consequence of Chinese remainder theorem. We give the details for completeness. Let $4lf^2=\prod_{j=1}^r(4lf^2)_{(q_j)}$ be the prime factorization of $4lf^2$. Then, by the Chinese remainder theorem, the map
\begin{align*}
\varphi&:\prod_{j=1}^r \mathbb{Z}/(4lf^2)_{(q_j)}\mathbb{Z}\rightarrow \mathbb{Z}/4lf^2\mathbb{Z}\\
&\hspace{-0.5in}(a_1,a_1,\cdots,a_r)\xrightarrow{\varphi} \sum_{j=1}^ra_{j}(4lf^2)^{(q_j)}((4lf^2)^{(q_j)})^{-1}
\end{align*}
is an isomorphism. Note also that $\varphi(a_{1},\cdots,a_r)\equiv a_{j}\bmod (4lf^2)_{(q_j)}$. Therefore,
\begin{align*}
Kl_{l,f}(\xi,n)&=\sum_{j=1}^r\sum_{\substack{a_{j}\bmod (4lf^2)_{(q_j)}\\ a_{j}^2-4n\equiv 0\bmod f_{(q_j)}^2\\ \frac{a_{j}^2-4n}{f_{(q_j)}^2}\equiv0,1\bmod 4_{(q_j)}}}e\left(\tfrac{\varphi(a_{1},\cdots,a_{r})\xi}{4lf^2}\right)\prod_{j=1}^r\left(\tfrac{(\varphi(a_{1},\cdots,a_{r})^2-4n)/f^2}{l_{(q_j)}}\right)\\
&=\sum_{j=1}^r\sum_{\substack{a_{j}\bmod (4lf^2)_{(q_j)}\\ a_{j}^2-4n\equiv 0\bmod f_{(q_j)}^2\\ \frac{a_{j}^2-4n}{f_{(q_j)}^2}\equiv0,1\bmod 4_{(q_j)}}}e\left(\tfrac{( \sum_{j=1}^ra_{j}(4lf^2)^{(q_j)}((4lf^2)^{(q_j)})^{-1})\xi}{4lf^2}\right)\prod_{j=1}^r\left(\tfrac{(a_{j}^2-4n)/f_{(q_j)}^2}{l_{(q_j)}}\right)
\end{align*}
\begin{align*}
&=\sum_{j=1}^r\sum_{\substack{a_{j}\bmod (4lf^2)_{(q_j)}\\ a_{j}^2-4n\equiv 0\bmod f_{(q_j)}^2\\ \frac{a_{j}^2-4n}{f_{(q_j)}^2}\equiv0,1\bmod 4_{(q_j)}}}\prod_{j=1}^r\left(\tfrac{(a_{j}^2-4n)/f_{(q_j)}^2}{l_{(q_j)}}\right)e\left(\tfrac{a_{j}((4lf^2)^{(q_j)})^{-1}\xi}{(4lf^2)_{(q_j)}}\right)\\
&=\prod_{j=1}^rKl_{l_{(q_j)},f_{(q_j)}}((4lf^2)^{(q_j)})^{-1}\xi,n).
\end{align*}

\end{proof}

By lemma \ref{estlem1}, we are reduced to analyzing $Kl_{q^{k_1},q^{k_2}}(\xi,n)$ for primes $q$.

\subsubsection{Local calculations}\label{apseclocal}

In this section we will explicitly calculate the local sums $Kl_{q^{k_1},q^{k_2}}(\xi,n)$. We will give complete details for odd primes $q$ below. The calculations for $q=2$ follow the same lines with more bookkeeping and we will leave the details of that case to the reader since we will not be needing the exact form of the answer (cf. corollary \ref{charsumcor2}). 

\begin{lemma}\label{estlem2}Let $q\neq 2$ be a prime. Then for any $\xi,n\in\mathbb{Z}$,
\begin{align*}
Kl_{q,1}(\xi,n)&=\begin{cases}q-1&\text{$v_q(\xi)\geq1$ and $v_q(4n)\geq1$ }\\ -1& \text{$v_q(\xi)\geq1$ and $v_q(4n)=0$}\\ S(\bar{2}\xi,2\xi n;q)& v_q(\xi)=0\end{cases}.
\end{align*}

\end{lemma}

\begin{proof} 
There are two cases depending on $v_q(\xi)\geq1$ or $v_q(\xi)=0$.

\begin{itemize}
\item $v_q(\xi)\geq1$. In this case $e\left(\tfrac{a\xi}{q}\right)=1$, so the sum reduces to
\[\sum_{a\bmod q}\left(\tfrac{a^2-4n}{q}\right)=\begin{cases}q-1&\text{$v_q(4n)\geq1$} \\ -1&\text{$v_q(4n)=0$}\end{cases}.\]

\item $v_q(\xi)=0$. For $q$ such that $\gcd(n,q)=1$, this is the statement of \cite{Sarnak:2001aa} equation (70)). When $q\mid n$ the sum is $\displaystyle\sum_{a\bmod^{\times }q}e\left(\tfrac{a\xi}{q}\right)=-1=S(\bar{2}\xi,0;q)$.

\end{itemize}
\end{proof}

\begin{lemma}\label{estlem3}Let $q\neq 2$ be a prime. Then for any $\xi,n\in\mathbb{Z}$ and for any $k_1\in\mathbb{Z}_{\geq0}$, 
\begin{align*}
q^{1-k_1}Kl_{q^{k_1},1}(\xi,n)&=\begin{cases}q-\left(1+\left(\tfrac{4n}{q}\right)\right)& v_q(\xi)\geq k_1\text{ and }k_1\equiv0\bmod2\\ q-1& v_q(\xi)\geq k_1\,,\,v_q(4n)\geq1\text{ and }k_1\equiv1\bmod 2\\ -1 & v_q(\xi)\geq k_1\,,\,v_q(4n)=0\text{ and }k_1\equiv1\bmod 2\\ -\left(1+\left(\tfrac{4n}{q}\right)\right)\cos\left(\tfrac{2\pi \xi\sqrt{4n}}{q^{k_1}}\right)& v_q(\xi)=k_1-1\text{ and } k_1\equiv 0\bmod2 \\ S(\bar{2}\xi^{(q)} ,2\xi^{(q)} n;q)& v_q(\xi)=k_1-1\text{ and }k_1\equiv 1\bmod2\\ 0&otherwise\end{cases},
\end{align*}
where $\sqrt{4n}$ denotes a square root\footnote{Note that the expression is independent of the choice of the square root since cosine is an even function.} of $4n$ modulo $q$, when exists.

\end{lemma}
\begin{proof}The calculation is divided into two cases depending on the parity of $k_1$.

\begin{itemize}

\item $k_1\equiv0\bmod 2$. 
\begin{align*}
Kl_{q^{k_1,1}}(\xi,n)&=\sum_{a\bmod q^{k_1}}\left(\tfrac{a^2-4n}{q^{k_1}}\right)e\left(\tfrac{a\xi}{q^{k_1}}\right)\\
&=\sum_{\substack{a\bmod q^{k_1}\\ a^2\neq4n\bmod q}}e\left(\tfrac{a\xi}{q^{k_1}}\right)\\
&=\sum_{\substack{a_0\bmod q\\ a_0^2\neq4n\bmod q}}e\left(\tfrac{a_0\xi}{q^{k_1}}\right)\sum_{a_1\bmod q^{k_1-1}}e\left(\tfrac{a_1\xi}{q^{k_1-1}} \right)
\end{align*}
\begin{align*}
&=\sum_{\substack{a_0\bmod q\\ a^2\neq4n\bmod q}}e\left(\tfrac{a_0\xi}{q^{k_1}}\right)\begin{cases}q^{k_1-1}&\text{$v_q(\xi)\geq k_1-1$}\\0&\text{$v_q(\xi)<k_1-1$}\end{cases}\\
&=q^{k_1-1}\begin{cases}q-\left(1+\left(\tfrac{4n}{q}\right)\right)& v_q(\xi)\geq k_1\\ -\left(1+\left(\tfrac{4n}{q}\right)\right)\cos\left(\tfrac{2\pi \xi \sqrt{4n}}{q^{k_1}}\right)& v_q(\xi)=k_1-1\\ 0& otherwise\end{cases}.
\end{align*}

\item $k_1\equiv1\bmod 2$. 
\begin{align*}
Kl_{q^{k_1},1}(\xi,n)&=\sum_{a\bmod q^{k_1}}\left(\tfrac{a^2-4n}{q^{k_1}}\right)e\left(\tfrac{a\xi}{q^{k_1}}\right)\\
&=\sum_{a_0\bmod q}\left(\tfrac{a_0^2-4n}{q}\right)e\left(\tfrac{a_0\xi}{q^{k_1}}\right)\sum_{a_1\bmod q^{k_1-1}}e\left(\tfrac{a_1\xi}{q^{k_1-1}}\right)\\
&=\sum_{a_0\bmod q}\left(\tfrac{a_0^2-4n}{q}\right)e\left(\tfrac{a_0\xi}{q^{k_1}}\right)\begin{cases}q^{k_1-1}&\text{$v_q(\xi)\geq k_1-1$}\\0&\text{$v_q(\xi)<k_1-1$}\end{cases}\\
&=q^{k_1-1}\begin{cases}q-1& v_q(\xi)\geq k_1\text{ and } v_q(4n)\geq1 \\ -1 & v_q(\xi)\geq k_1\text{ and } v_q(4n)=0\\ S(\bar{2}\xi q^{1-k_1},2\xi q^{1-k_1}n;q)& v_q(\xi)=k_1-1\\ 0&otherwise
\end{cases},
\end{align*}
where we used lemma \ref{estlem2} to get the last equality. The lemma now follows from the equality $\xi^{(q)}=\xi q^{1-k_1}$ in the third case above.

\end{itemize}

\end{proof}

\begin{lemma}\label{estlem4}Let $q\neq2$ be a prime. Then, for any $\xi,n\in\mathbb{Z}$ and for any $k_2\in\mathbb{Z}_{\geq0}$
\begin{equation}
q^{-\frac{\min\{v_q(n),2k_2\}}{2}}Kl_{1,q^{k_2}}(\xi,n)=\delta(4n;q^{2k_2})\begin{cases}1&\text{$v_q(\xi)\geq k_2$, $v_q(n)\geq 2k_2$}\\ 2\cos\left(\tfrac{2\sqrt{n}\xi}{q^{2k_2}}\right)&\text{$2v_q(\xi)\geq v_q(n)$, $v_q(n)<2k_2$ }\\ 0&otherwise\end{cases}.
\end{equation}

\end{lemma}

\begin{proof}By definition,

\begin{equation*}
Kl_{1,q^{k_2}}(\xi,n)=\sum_{\substack{a\bmod q^{2k_2}\\ a^2\equiv4n\bmod q^{2k_2}}}e\left(\tfrac{a\xi}{q^{2k_2}}\right).
\end{equation*}
(Note that since $q\equiv 1\bmod 2$ the condition that $\tfrac{a^2-4n}{q^{2k_2}}\equiv 0,1\bmod 4$ is vacuous.) The sum is trivially $1$ when $k_2=0$. Assume that $k_2>0$, then the sum is $0$ unless $4n$ is a square $\bmod\,\, q^{2k_2}$. If $4n$ is a square we have two sub-cases depending on $v_q(4n) \geq2 k_2$ or not. 
\begin{itemize}
\item $v_q(4n)\geq 2k_2$. In this case $a^2\equiv 4n\equiv 0\bmod q^{2k_2}$ which implies (since $q\neq2$) $a\equiv 0\bmod q^{k_2}$. Therefore the sum is,
\begin{align*}
Kl_{1,q^{k_2}}&=\sum_{\substack{a\bmod q^{2k_2}\\ a\equiv0\bmod q^{k_2}}}e\left(\tfrac{a\xi}{q^{2k_2}}\right)\\
&=\sum_{\substack{a_0\bmod q^{k_2}}}e\left(\tfrac{a_0\xi}{q^{k_2}}\right)\\
&=\begin{cases}q^{k_2}&\text{$v_q(\xi)\geq k_2$}\\ 0&\text{$v_q(\xi)<k_2$} \end{cases}.\tag{i}\label{esti}
\end{align*}

\item $v_q(4n)<2k_2$. Let $n=q^{2r}n^{(q)}$ (Note that $v_q(n)$ is even since otherwise the sum is necessarily $0$.). Then, $a^2\equiv 2n\bmod q^{2k_2}$ if and only if $a=q^{r}a_0$ for some $a_0\bmod q^{2k_2-r}$ where $a_0^2\equiv 4n^{(q)}\bmod q^{2k_2-2r}$. Finally $a_0^2\equiv 4n^{(q)}\bmod q^{2k-2r}$ if and only if (note that $q\neq2$) $a_0=\pm2\sqrt{n^{(q)}}+a_1q^{2k-2r}$ for some $a_1\bmod q^r$, where by abuse of notation we use $\sqrt{n^{(q)}}$ to denote a square root of $n_0$ modulo $q^{2k_2-r}$. Hence,
\begin{align*}
Kl_{1,q^{k_2}}&=\sum_{\substack{a_0\bmod q^{2k_2-r}\\ a_0^2\equiv 4n^{(q)}\bmod q^{2k_2-2r}}}e\left(\tfrac{a_0\xi}{q^{2k_2-r}}\right)\\
&=2\cos\left(\tfrac{2\sqrt{n^{(q)}}\xi}{q^{2k_2-r}}\right)\sum_{\substack{a_1\bmod p^{r}}}e\left(\tfrac{a_1\xi}{q^r}\right)\\
&=2\cos\left(\tfrac{2\sqrt{n}\xi}{q^{2k_2}}\right)\begin{cases}q^{\frac{v_q(n)}{2}}&\text{$2v_q(\xi)\geq v_q(n)$}\\0&\text{$2v_q(\xi)<v_q(n)$}\tag{ii}\label{estii}\end{cases}.
\end{align*}

\end{itemize}

Combining \eqref{esti} and \eqref{estii} gives the result.

\end{proof}

\begin{lemma}\label{estlem5}Let $p\neq2$ be a prime. Then, for any $\xi,n\in\mathbb{Z}$ and for any $k_1,k_2\in\mathbb{Z}_{\geq1}$, the value of 
\[q_v^{1-k_1-\frac{\min\{2k_2,v_q(n)\}}{2}}Kl_{q^{k_1},q^{k_2}}(\xi,n)\]
is given by the following:
\begin{itemize}
\item If $v_q(4n)\geq 2k_2$,

\begin{equation*} 
\delta(4n;q^{2k_2})\begin{cases} q-\left(1+\left(\tfrac{4n^{(q)}}{q}\right)\right) & v_q(\xi)\geq k_1+k_2,\,\, k_1\equiv 0\bmod 2\\  q-1&\text{$v_q(\xi)\geq k_1+k_2$, $v_q(4n)=2k_2+1$, and $k_1\equiv1\bmod 2$}\\ -1& \text{$v_q(\xi)\geq k_1+k_2$, $v_q(4n)=2k_2$, and $k_1\equiv 1\bmod2$}\\ -\left(1+\left(\tfrac{4n^{(q)}}{q}\right)\right)\cos\left(\tfrac{2\pi \xi\sqrt{4n}}{q^{k_1+2k_2}}\right)&\text{$v_q(\xi)=k_1+k_2-1$, $k_1\equiv 0\bmod 2$ }\\ S(\bar{2}\xi^{(q)},2\xi^{(q)}nq^{-2k_2};q)& \text{$v_q(\xi)=k_1+k_2-1$, $k_1\equiv1\bmod2$}\\0&otherwise\end{cases}
\end{equation*}
\item If $v_q(4n)<2k_2$,
\[\delta(4n;q^{2k_2})\begin{cases}2(q-1)\cos\left(\tfrac{\sqrt{4n}\xi}{q^{k_1+2k_2}}\right)&\text{$v_q(\xi)\geq k_1+r$, $k_1\equiv0\bmod2$}\\0&\text{$v_q(\xi)\geq k_1+r$, $k_1\equiv 1\bmod 2$} \\ -2\cos\left(\tfrac{\sqrt{4n}\xi}{q^{k_1+2k_2}}\right) &\text{$v_q(\xi)=k_1+r-1$, $k_1\equiv 0\bmod2$}\\  2\sqrt{q}\left(\tfrac{4\xi^{(q)}\sqrt{n^{(q)}}}{q}\right)\cos\left(\tfrac{\sqrt{4n}\xi}{q^{k_1+2k_2}}\right)&\text{$v_q(\xi)=k_1+r-1$, $k_1\equiv1\bmod2$, and $q\equiv 1\bmod 4$}\\  -2\sqrt{q}\left(\tfrac{4\xi^{(q)}\sqrt{n^{(q)}}}{q}\right)\sin\left(\tfrac{\sqrt{4n}\xi}{q^{k_1+2k_2}}\right)&\text{$v_q(\xi)=k_1+r-1$, $k_1\equiv1\bmod2$, and $q\equiv 3\bmod 4$}\\ 0& otherwise\end{cases}\]

\end{itemize}

\end{lemma}

\begin{proof}

The sum vanishes unless $4n$ is a square $\bmod q^{2k_2}$. Assuming this let $a=a_0+a_1q^{2k_2+1}$. Then,
\begin{align*}
Kl_{q^{k_1},q^{k_2}}(\xi,n)&=\sum_{\substack{a_0\bmod q^{2k_2+1}\\ a_0^2\equiv 4n\bmod q^{2k_2}}}\left(\tfrac{(a_0^2-4n)/q^{2k_2}}{q^{k_1}}\right)e\left(\tfrac{a_0\xi}{q^{k_1+2k_2}}\right)\sum_{a_1\bmod q^{k_1-1}}e\left(\tfrac{a_1\xi}{q^{k_1-1}}\right)\\
&=\sum_{\substack{a_0\bmod q^{2k_2+1}\\ a_0^2\equiv 4n\bmod q^{2k_2}}}\left(\tfrac{(a_0^2-4n)/q^{2k_2}}{q^{k_1}}\right)e\left(\tfrac{a_0\xi}{q^{k_1+2k_2}}\right)\begin{cases}q^{k_1-1}&\text{$v_q(\xi)\geq k_1-1$}\\ 0&\text{$v_q(\xi)<k_1-1$}\end{cases}.\tag{i}\label{estiii}
\end{align*}
Therefore the sum vanishes unless $v_q(\xi)\geq k_1-1$. For the rest of the analysis we assume that the sum does not vanish. We have two cases according to $v_q(4n)\geq 2k_2$ or $v_q(4n)<2k_2$. 

\begin{itemize}
\item $v_q(4n)\geq 2k_2$. In this case we need to have $a_0^2\equiv 0\bmod q^{2k_2}$ therefore $a_0\equiv 0\bmod q^{k_2}$. Hence the sum in \eqref{estiii} is,
\[q^{k_1-1}\sum_{\substack{a_0\bmod q^{2k_2+1}\\ a_0\equiv 0\bmod q^{k_2}}}\left(\tfrac{(a_0^2-4n)/q^{2k_2}}{q^{k_1}}\right)e\left(\tfrac{a_0\xi}{q^{k_1+2k_2}}\right)=q^{k_1-1}\sum_{\substack{a_2\bmod q^{k_2+1}}}\left(\tfrac{a_2^2-4nq^{-2k_2}}{q^{k_1}}\right)e\left(\tfrac{a_2\xi}{q^{k_1+k_2}}\right)\]
Using lemma \ref{estlem3} on the last sum we get the result.

\item $v_q(4n)<2k_2$. We are summing over $a_0\in\mathbb{Z}/q^{2k_2+1}\mathbb{Z}$ that satisfy $a_0^2-4n\equiv0\bmod q^{2k_2}$. Note that the sum is $0$ if $v_q(n)\equiv 1\bmod 2$, and let $n=q^{2r}n^{(q)}$. Then $a_0^2\equiv 4n\bmod q^{2k_2}$ if and only if $a_0=2q^ra_1$ for some $a_1\in\mathbb{Z}/q^{2k_2-r+1}\mathbb{Z}$ which satisfies $a_1^2\equiv n^{(q)}\bmod q^{2k_2-2r}$. Therefore the sum in \eqref{estiii} can be written as
\begin{equation*}
\sum_{\substack{a_0\bmod q^{2k_2+1}\\ a_0^2\equiv 4n\bmod q^{2k_2}}}\left(\tfrac{(a_0^2-4n)/q^{2k_2}}{q^{k_1}}\right)e\left(\tfrac{a_0\xi}{q^{k_1+2k_2}}\right)
=\sum_{\substack{a_1\bmod q^{2k_2-r+1}\\ a_1^2\equiv 4n^{(q)}\bmod q^{2k_2-2r}}}\left(\tfrac{(a_1^2-4n^{(q)})/q^{2k_2-2r}}{p^{k_1}}\right)e\left(\tfrac{a_1\xi}{q^{k_1+2k_2-r}}\right)
\end{equation*}
Let $a_1=a_2+a_3q^{2k_2-2r+1}$, where $a_2\in\mathbb{Z}/q^{2k_2-2r+1}\mathbb{Z}$ and $a_3\in\mathbb{Z}/q^{r}\mathbb{Z}$. Substituting this in the above sum we rewrite it as,
\begin{align*}
&\sum_{\substack{a_2\bmod q^{2k_2-2r+1}\\ a_2^2\equiv 4n^{(q)}\bmod q^{2k_2-2r}}}\left(\tfrac{(a_2^2-4n^{(q)})/q^{2k_2-2r}}{p^{k_1}}\right)e\left(\tfrac{a_2\xi}{q^{k_1+2k_2-r}}\right)\sum_{a_3\bmod q^r}e\left(\tfrac{a_3\xi}{q^{k_1+r-1}}\right)\\
&=\sum_{\substack{a_2\bmod q^{2k_2-2r+1}\\ a_2^2\equiv 4n^{(q)}\bmod q^{2k_2-2r}}}\left(\tfrac{(a_2^2-4n^{(q)})/q^{2k_2-2r}}{p^{k_1}}\right)e\left(\tfrac{a_2\xi}{q^{k_1+2k_2-r}}\right)\begin{cases}q^{r}&\text{$v_q(\xi)\geq k_1+r-1$}\\ 0&\text{$v_q(\xi)<k_1+r-1$} \end{cases}\tag{ii}\label{estiv}
\end{align*}
Assuming that the sum doesn't vanish (i.e. $v_q(\xi)\geq k_1+r-1$), the only remaining part is to calculate the character sum above. To do so, note that the solutions to $a_0^2\equiv 4n^{(q)}\bmod q^{2k_2-2r}$ where $a_0\in\mathbb{Z}/q^{2k_2-2r+1}\mathbb{Z}$ are all of the form $\pm2\sqrt{n^{(q)}}+a_4q^{2k_2-2r}$, where by abuse of notation we denote\footnote{To keep uniformity, one can pick a $q$-adic solution (which exists by Hensel's lemma) at once and consider its reductions. Note that the sum is independent of  all the choices that are made.} a solution to $a_0^2\equiv4n^{(q)}\bmod q^{2k_2-2r+1}$ by $2\sqrt{n^{(q)}}$, and $a_4\in\mathbb{Z}/q\mathbb{Z}$. Using this observation, we rewrite the character sum above as,
\[\sum_{\substack{a_4\bmod q\\ \pm}}\left(\tfrac{\pm 4\sqrt{n^{(q)}}a_4}{q^{k_1}}\right)e\left(\tfrac{\pm2\sqrt{n^{(q)}}\xi}{q^{k_1+2k_2-r}}\right)e\left(\tfrac{a_4\xi}{q^{k_1+r}}\right).\]
Since $v_q(n^{(q)})=0\Rightarrow v_q(\sqrt{n^{(q)}}=0)$, and $q\neq2$ we can further rewrite this as
\[\left(\tfrac{4\xi^{(q)}\sqrt{n^{(q)}}}{q^{k_1}}\right)\left[e\left(\tfrac{\sqrt{4n}\xi}{q^{k_1+2k_2}}\right)+\left(\tfrac{-1}{q^{k_1}}\right)e\left(\tfrac{-\sqrt{4n}\xi}{q^{k_1+2k_2}}\right)\right]\sum_{\substack{a_4\bmod q}}\left(\tfrac{a_4}{q^{k_1}}\right)e\left(\tfrac{a_4\xi_{(q)}}{q^{k_1+r}}\right).\tag{iii}\label{estv}\]
Finally, 
\[\sum_{\substack{a_4\bmod q}}\left(\tfrac{a_4}{q^{k_1}}\right)e\left(\tfrac{a_4\xi_{(q)}}{q^{k_1+r}}\right)=\begin{cases}q-1&\text{$v_q(\xi)\geq k_1+r$, $k_1\equiv0\bmod2$}\\0&\text{$v_q(\xi)\geq k_1+r$, $k_1\equiv 1\bmod 2$} \\ -1 &\text{$v_q(\xi)=k_1+r-1$, $k_1\equiv 0\bmod2$}\\  \sqrt{q}&\text{$v_q(\xi)=k_1+r-1$, $k_1\equiv1\bmod2$, and $q\equiv 1\bmod 4$}\\   i\sqrt{q}&\text{$v_q(\xi)=k_1+r-1$, $k_1\equiv1\bmod2$, and $q\equiv 3\bmod 4$}\end{cases}.\]
Note that in the last two lines we used the explicit value of the Gauss sum (cf. (1.55) of \cite{Iwaniec:2004aa}). Substituting this in \eqref{estv} then gives,
\[\eqref{estv}=\begin{cases}2(q-1)\cos\left(\tfrac{\sqrt{4n}\xi}{q^{k_1+2k_2}}\right)&\text{$v_q(\xi)\geq k_1+r$, $k_1\equiv0\bmod2$}\\0&\text{$v_q(\xi)\geq k_1+r$, $k_1\equiv 1\bmod 2$} \\ -2\cos\left(\tfrac{\sqrt{4n}\xi}{q^{k_1+2k_2}}\right) &\text{$v_q(\xi)=k_1+r-1$, $k_1\equiv 0\bmod2$}\\  2\sqrt{q}\left(\tfrac{4\xi^{(q)}\sqrt{n^{(q)}}}{q}\right)\cos\left(\tfrac{\sqrt{4n}\xi}{q^{k_1+2k_2}}\right)&\text{$v_q(\xi)=k_1+r-1$, $k_1\equiv1\bmod2$, and $q\equiv 1\bmod 4$}\\  -2\sqrt{q}\left(\tfrac{4\xi^{(q)}\sqrt{n^{(q)}}}{q}\right)\sin\left(\tfrac{\sqrt{4n}\xi}{q^{k_1+2k_2}}\right)&\text{$v_q(\xi)=k_1+r-1$, $k_1\equiv1\bmod2$, and $q\equiv 3\bmod 4$}\end{cases}\]
Combining  \eqref{estiii}, \eqref{estiv}, and \eqref{estv} gives the result.

\end{itemize}

\end{proof}

\subsubsection{Bounds}

\begin{cor}\label{charsumcor1}Let $q\neq2$ be a prime. Then, for any $k_1,k_2\in\mathbb{Z}_{\geq0}$,
\[Kl_{q^{k_1},q^{k_2}}=\delta(4n;q^{2k_2})\begin{cases}O\left(q^{k_1}\sqrt{\gcd(q^{2k_2},n)}\right)& q^{2k_1}\gcd(q^{2k_2},n)\mid \xi^2\\O\left(\frac{q^{k_1}\sqrt{\gcd(q^{2k_2},n)}}{\sqrt{q}}\right)& q^{2(k_1-1)}\gcd(q^{2k_2},n)\parallel \xi^2 \\ 0& otherwise\end{cases},\]
where the implied constants are absolute.
\end{cor}

\begin{proof}The only non-trivial input we use is the Weil bound on Kloosterman sums (\cite{Weil:1948aa}), which comes in for the second line above. The bound states that (cf. (1.60) of \cite{Iwaniec:2004aa} for the statement we are using), 
\[|S(a,b;q)|\leq 2\sqrt{\gcd(a,b,q)}\sqrt{q}.\]
Since $\xi^{(q)}$ is relatively prime to $q$, this bound implies,
\[|S(\bar{2}\xi^{(q)},2\xi^{(q)}nq^{-2k_2};q)|\leq 2\sqrt{q}.\]
We then use this bound on the Kloosterman sums that appear in $v_q(\xi)=0$ of lemmas \ref{estlem2}, \ref{estlem3}, and \ref{estlem5}. Bounding the rest of the terms trivially proves the corollary.

\end{proof}

We finally remark that similar calculations in lemmas \ref{estlem2} to \ref{estlem5} leads to the same bound for $q=2$. Since we will not be needing the exact form of $Kl_{2^{k_1},2^{k_2}}(\xi,n)$, we only state the relevant bound.

\begin{cor}\label{charsumcor2}For any $k_1,k_2\in\mathbb{Z}_{\geq0}$,
\[Kl_{2^{k_1},2^{k_2}}=\delta(4n;2^{2k_2})\begin{cases}O\left(2^{k_1}\sqrt{\gcd(2^{2k_2},n)}\right)& 2^{2k_1}\gcd(2^{2k_2},n)\mid \xi^2\\O\left(\frac{2^{k_1}\sqrt{\gcd(2^{2k_2},n)}}{\sqrt{2}}\right)& 2^{2(k_1-1)}\gcd(2^{2k_2},n)\parallel \xi^2 \\ 0& otherwise\end{cases},\]
where the implied constants are absolute.
\end{cor}

\begin{proof}
The explicit calculations for the character sum are identical to the ones given for lemmas \ref{estlem2} to \ref{estlem5}. The only difference is that one now needs to take into account the requirement $\frac{a^2-4n}{q^{2k_2}}\equiv 0,1\bmod4$. We also need to recall that the Kronecker symbol, $\left(\frac{\alpha}{2}\right)$, is periodic in $\alpha$ modulo $8$, and that for $\beta\in \left(\mathbb{Z}/2^k\mathbb{Z}\right)^{\times}$ the number of solutions to $x^2\equiv \beta^2 \bmod\,\, 2^k$ are $1,2$, or $4$ depending on $k=1,2$, or $k\geq3$. The result then follows from a case by case analysis of $v_2(n)\geq 2k_2$, $v_q(n)=2k_2-2,2k_2-4$, and $v_2(n)\leq 2k_2-6$ and bookkeeping. We also note that in this case we do not even need to appeal to the Weil bound since we can explicitly calculate the sums $Kl_{2^{k_1},1}(\xi,4n)$, for $k_1=1,2,3$, in this case. We leave the details to the reader.

\end{proof}

\begin{cor}\label{charsumcor3} Let $l,f\in\mathbb{Z}_{\geq1}$. For any $n,\xi\in\mathbb{Z}$,
\[Kl_{l,f}(\xi,n)\ll\delta(n;f^2)\begin{cases}\log(lf^2)\sqrt{l\gcd(n,f^2)}\sqrt{\gcd\left(\tfrac{\xi}{\sqrt{\gcd(n,f^2)}},l\right)}&\tfrac{l\sqrt{\gcd(n,f^2)}}{rad(l)}\mid \xi\\ 0&otherwise\end{cases},\]
where $rad(l)=\prod_{q\mid l}q$ denotes the radical of $l$.
\end{cor}

\begin{proof}Let $4lf^2\prod_{q}q^{k_1+2k_2}$. Note that by corollaries \ref{charsumcor1} and \ref{charsumcor2}, for any $\alpha\in\mathbb{Z}$ with $\gcd(\alpha,lf^2)=1$ we have $|Kl_{q^{k_1},q^{k_2}}(\alpha \xi,n)|=|Kl_{q^{k_1},q^{k_2}}(\xi,n)|$. Then, by lemma \ref{estlem1} $|Kl_{l,f}(\xi,n)|=\prod_q|Kl_{q^{k_1},q^{k_2}}(\xi,n)|$. The result now follows from corollaries \ref{charsumcor1} and \ref{charsumcor2}, and the observation that $\prod_{q\mid lf^2}O(1)=O(\log(lf^2))$.

\end{proof}

\bibliographystyle{alpha}

\begin{thebibliography}{Kna01}

\bibitem[Alt15]{Altug:2015aa}
S.~A. Altu\u{g}.
\newblock Beyond {E}ndoscopy via the {T}race {F}ormula-{I}: {P}oisson
  {S}ummation and {I}solation of {S}pecial {R}epresentations.
\newblock {\em Compositio Mathematica}, 2015.

\bibitem[Bou04]{Bourbaki:2004aa}
N.~Bourbaki.
\newblock {\em Functions of a real variable}.
\newblock Elements of mathematics. Springer-{V}erlag, Berlin, Heidelberg,
  Germany, 2004.

\bibitem[Del74]{Deligne:1974aa}
P.~Deligne.
\newblock La conjecture de {W}eil {I} (french).
\newblock {\em Inst. {H}autes {\'E}tudes {S}ci. {P}ubl. {M}ath.},
  \textbf{43}:273--307, 1974.

\bibitem[IK04]{Iwaniec:2004aa}
H.~Iwaniec and E.~Kowalski.
\newblock {\em Analytic Number Theory}, volume~53 of {\em Colloquium
  Publications}.
\newblock AMS, Providence, RI, 2004.

\bibitem[JL70]{Jacquet:1970aa}
H.~Jacquet and R.~P. Langlands.
\newblock {\em Automorphic forms on $GL(2)$}, volume \textbf{114} of {\em SLM}.
\newblock Springer-{V}erlag, Berlin-Heidelberg, New York, 1970.

\bibitem[Klo27]{Kloosterman:1927aa}
H.~D. Kloosterman.
\newblock Asymptotische {F}ormeln f\"ur die {F}ourierkoeffizienten ganzer
  {M}odulformen. ({G}erman).
\newblock {\em Abh. {M}ath. {S}em. {U}niv. {H}amburg}, (1):337--352, 1927.

\bibitem[Kna01]{Knapp:2001aa}
A.~W. Knapp.
\newblock {\em Representation {T}heory of {S}emisimple {G}roups: {A}n
  {O}verview {B}ased on {E}xamples}, volume \textbf{(PMS-36)} of {\em Princeton
  {L}andmarks in {M}athematics}.
\newblock Princeton University Press, Princeton, NJ, 2001.

\bibitem[Kuz80]{Kuznetsov:1980aa}
N.~V. Kuznetsov.
\newblock The {P}etersson conjecture for cusp forms of weight zero and the
  {L}innik conjecture. {S}ums of {K}loosterman sums. ({R}ussian).
\newblock {\em Mat. Sb. (N.S.)}, \textbf{111} (153)(3):334--383, 479, 1980.

\bibitem[Lan04]{Langlands:2004aa}
R.~P. Langlands.
\newblock Beyond {E}ndoscopy.
\newblock In {\em Contributions to Automorphic Forms, Geometry, and Number
  Theory}, chapter~22, pages 611--698. The Johns Hopkins University Press,
  Baltimore, MD, 2004.

\bibitem[Lan13]{Langlands:2013aa}
R.~P. Langlands.
\newblock Singularit\'es et {T}ransfert.
\newblock {\em Annales math\'ematiques du Qu\'ebec}, \textbf{37}(2):173--253,
  2013.

\bibitem[Mor77]{Moreno:1977aa}
C.~J. Moreno.
\newblock A quick proof of {H}ardy's estimate for {R}amanujan function
  $\tau(n)$.
\newblock {\em Journal of {N}umber {T}heory}, \textbf{9}:1--3, 1977.

\bibitem[Sar01]{Sarnak:2001aa}
P.~Sarnak.
\newblock Comments on {R}obert {L}angnlands' {L}ecture: ``{E}ndoscopy and
  {B}eyond".
\newblock Available at:
  http://publications.ias.edu/sites/default/files/SarnakLectureNotes-1.pdf,
  2001.

\bibitem[Sar05]{Sarnak:2005aa}
P.~Sarnak.
\newblock Note on the generalized {R}amanujan conjectures.
\newblock In {\em Harmonic {A}nalysis, the {T}race {F}ormula, and {S}himura
  {V}arieties}, volume \textbf{4} of {\em Clay Mathematics Proceedings}, pages
  659--685. CMI, Providence, RI, 2005.

\bibitem[She79]{Shelstad:1979aa}
D.~Shelstad.
\newblock Orbital integrals for ${GL}_2(\mathbb{R})$.
\newblock In {\em Automorphic Forms, Representations, and $L$-functions},
  volume \textbf{33}, part I of {\em Proc. Symp. Pure Math.}, pages 107--110.
  AMS, Providence, RI, 1979.

\bibitem[Wei48]{Weil:1948aa}
A.~Weil.
\newblock On some exponential sums.
\newblock {\em Proc. Natl. Acad. Sci. USA}, \textbf{34}:204--207, 1948.

\end{thebibliography}

\end{document}